\documentclass[11pt, reqno]{amsart}

\usepackage{amsfonts, amssymb, amscd, amsrefs}
\usepackage{graphicx}
\usepackage{hyperref}
\usepackage{slashed}
\usepackage{fullpage}

\usepackage{graphicx} 
\usepackage{mathtools}
\usepackage{pgfplots}

\usepackage{bbold}
\makeatletter
\def\amsbb{\use@mathgroup \M@U \symAMSb}
\makeatother

\newtheorem{theorem}{Theorem}
\newtheorem{rem}{Remark}

\newtheorem{defn}{Definition}

\newtheorem{opt}{Optimization Problem}
\newtheorem{lem}{Lemma}
\newtheorem{cor}{Corollary}

\newcommand{\E}{\amsbb{E}}
\newcommand{\tB}{\Tilde{B}}
\newcommand{\tX}{\Tilde{X}}
\newcommand{\cE}{\mathcal{E}}
\newcommand{\cN}{\mathcal{N}}
\newcommand{\R}{\amsbb{R}}
\newcommand{\loc}{\mathrm{loc}}

\title{Optimal two-parameter portfolio management strategy with transaction costs}

\begin{document}

\author{Chutian Ma}
       \address[Chutian Ma]{Department of Mathematics, Johns Hopkins University}
       \email{cma27@jhu.edu}

\author{Paul Smith}
       \address[Paul Smith]{Causify.AI Inc.}
       \email{paul@causify.ai}
\maketitle

\begin{abstract}
We consider a simplified model for optimizing a single-asset portfolio in the
presence of transaction costs given a signal with a certain autocorrelation
and cross-correlation structure. In our setup, the portfolio manager is given
two one-parameter controls to influence the construction of the portfolio.
The first is a linear filtering parameter that may increase or decrease the
level of autocorrelation in the signal. The second is a numerical threshold
that determines a symmetric ``no-trade" zone. Portfolio positions are
constrained to a single unit long or a single unit short. These constraints
allow us to focus on the interplay between the signal filtering mechanism
and the hysteresis introduced by the ``no-trade" zone.
We then formulate an optimization problem where we aim to minimize the
frequency of trades subject to a fixed return level of the portfolio.
We show that maintaining a no-trade zone while removing autocorrelation
entirely from the signal yields a locally optimal solution.
For any given ``no-trade" zone threshold, this locally optimal solution
also achieves the maximum attainable return level, and we derive a
quantitative lower bound for the amount of improvement in
terms of the given threshold and the amount of autocorrelation removed.
\end{abstract}

\tableofcontents

\section{Introduction}
We consider an optimization problem for a discrete-time stochastic system
motivated by a portfolio construction and optimization question that arises in
practice in quantitative finance in the presence of transaction costs. Suppose
a portfolio manager is in possession of a signal that is positively correlated
with the immediate subsequent price movement of the asset but uncorrelated
otherwise.  The manager seeks to build his portfolio so that the expected
return is maximized over the trading period. A greedy approach would be to set
the portfolio proportional to the signal strength at each stage.  However, in
the presence of transaction costs, trading frequently is costly, incurring
slippage, and diminishes the total return. Thus, the portfolio manager must
weigh the potential gain and slippage against each other in order to obtain an
optimal balance. In this paper, we formulate the preceding question in an
optimization problem where the manager seeks to minimize the number of trades,
thus minimizing the impact of transaction costs, under the constraint of a
fixed return level.

In the absence of transaction costs, the portfolio optimization problem was
first studied by Merton in \cite{merton1969lifetime} and
\cite{merton1975optimum} in the continuous time seetting. Without transaction
costs, the optimal policy is to update the portfolio to the optimal state
predicted by the signal at each period, which is called ``Markowitz portfolio".
However, in the presence of transaction costs, this is often not optimal, as
trading into a large position, for example, will incur future costs when it is
time to unwind the trade. In fact, Magill and Constantinides show in
\cite{magill1976portfolio} that the optimal policy in the presence of
transaction costs takes the form of a ``no-trade zone". They proved in the
multi-asset context that the investor should keep the proportion of the wealth
invested in various assets inside a certain region. In other words, if the
investor starts with his portfolio inside the no-trade zone, a future
transaction only occurs at the boundary of the no-trade zone. In the context of
single-asset portfolio, the policy of a no-trade zone simply means that we only
trade when the predictive signal exceeds a specified threshold.

Since the work of Magill and Constantinides, various results have been
established regarding the properties and characterization of the no-trade zone
policy. For example, Constantinides in \cite{constantinides1979multiperiod}
gave a detailed analysis of the shape and other analytic properties of the
no-trade zone. More recently, \cite{DeMeNo16} derived a closed form expression
for the no-trade zone.

An alternative approach to reduce the slippage from transaction costs is to
``slow down" the trades. Nicolae G\^arleanu and Lasse Heje Pedersen in their
paper \cite{garleanu2013dynamic} showed that rather than pursuing the Markowitz
portfolio at each period, the portfolio manager should instead keep his
portfolio as a linear combination of the current portfolio and an ``aim"
portfolio, which the authors summarize as ``trade partially toward the target".
In this paper, we propose an alternative approach to slow down trading. Namely,
the portfolio manager may choose to slow down the signal instead of the
portfolio via the application of moving averages. Mathematically, a moving
average is a convolutional operator.  Thus it is able to reduce the amount of
erroneous trading by de-noising the signal. See \cite{genccay2001introduction}
Chapter 3 for a detailed discussion.  In addition to its de-noising properties,
moving averages are widely used to extract information from financial time
series. Zumbach and M\"{u}ller in \cite{zumbach2001operators} treated
inhomogeneous time series with moving averages. They developed a framework so
that a set of basic moving averages can be combined to estimate more
sophisticated quantities, such as different kinds of volatility and
correlation. For earlier work that deal with homogeneous time series, we refer
the reader to \cite{granger2014forecasting}, \cite{priestley1988non} and
\cite{hamilton2020time}.  Among various filters, the exponential moving average
is a simple yet successful one and indeed has many optimal properties in
forecasting. See \cite{muth1960optimal} for a general discussion and
\cite{raudys2013moving} for its application in finance data smoothing. In
general, passing the signal through filters helps de-noise the signal, which in
turn reduces trading and hence the cumulative transaction costs.

In this paper we aim to analyze the portfolio management problem of a single
asset with both a no-trade zone and a moving average incorporated into our
strategy. We will study the interaction between the two methods and find the
optimal no-trade zone and the optimal amount of smoothing that together
minimize the trading frequency subject to a fixed return level constraint. The
problem can be formulated as an optimization problem, with two parameters
$\alpha$ and $\eta$, where $\eta$ controls the size of the no-trade zone and
$\alpha$ controls the strength of smoothing. We prove that using a no-trade
zone alone is sufficient to yield the locally optimal solution. In other words,
if the portfolio manager starts with an autocorrelated signal, it is more
effective to remove the autocorrelation and rely on the no-trade zone alone in
portfolio construction.
\begin{rem}
    Note that heuristically our result is consistent with the Markov nature of
    the problem, i.e., the reward function depends only on the current state.
    One will indeed expect the optimal policy to use the information from
    the current period only (see \cite{bertsekas1996stochastic} for a detailed
    discussion).
\end{rem}

\section{Problem Formulation}
In this paper, we consider a simplified case where the portfolio consists of a
single asset and we allow only two positions: namely go long by 1 unit or short
by 1 unit. The policy of having a no-trade zone is then reduced to setting
symmetric thresholds and switching positions only when the signal crosses the
opposing threshold.

Informally, we have access to a real-valued random \emph{signal} $X_t$ at
time $t$ that is correlated with a real-valued random \emph{return} $Y_t$.
Morever, we assume that $X_t$ is uncorrelated with $Y_s$ for $t \neq s$.
At time $t$, we observe $X_t$ and then determine a position $w_t$. After
entering position $w_t$, we realize the \emph{reward} $w_t \cdot Y_t$.
If $w_t$ differs from $w_{t - 1}$ (we discuss initialization in more detail
below), then we incur some cost, e.g., a \emph{transaction cost}.
So far, if we introduce some assumptions on the moments of $Y_t$ (e.g., the
first and second moments exist and are finite), this corresponds to a
simplified version of a standard portfolio optimization problem.

We simultaneously consider two extensions of this standard problem.
First, we assume some additional structure on the signal $X_t$ over time,
either given a priori or determined through the selection from a 1-parameter
family of linear filters. Second, we consider the reward/cost trade-off over
the long-term (multiple periods) in that we consider average rewards and
average costs.

More formally, let $\{X_t\} \sim \cN(0, 1)$, $t = 0, 1, 2, \ldots$, denote a
sequence of independent identically distributed (\emph{iid}) Gaussian random
variables. We refer to this sequence of random variables as the signal. Let
$\{\varepsilon_t\} \sim \cN(0, 1)$, $t = 0, 1, 2, \ldots$ denote another
iid sequence of Gaussian random variables, independent from the first. We refer
to this as the noise sequence. Let $0 < \rho < 1$ be a fixed correlation
coefficient, and define the sequence $\{Y_t\}$ of random variables via
\begin{equation}\label{y_t}
    Y_t = \rho X_t + \sqrt{1 - \rho^2} \epsilon_t.
\end{equation}
We interpret $\{Y_t\}$ as the return of a financial instrument. Note that
we work with a standard random walk rather than an exponential random walk.
On short time scales, a standard random walk (or, in the continuous limit,
a standard Brownian motion) may be a more appropriate price model than an
exponential one (e.g., see the discussion in \cite{BoBoDoGo18}[\S 2.1.1]).

\subsection{Constraining the signal}

Above we introduced the signal $X_t$ as a sequence of iid Gaussians. An
important and naturally motivated class of signals is the class of stable first
order linear autoregressive processes, i.e., stable AR(1) processes. The
continuous analogue is the Ornstein-Uhlenbeck process (not considered in this
work). References for the role of AR(1) processes in time series modeling,
including how they fit into the broader ARIMA framework, structural time series
modeling, and Bayesian frameworks, include \cite{BoJe76}, \cite{Ha89}, and,
more recently, \cite{PrFeWe21}.

AR(1) models are a natural object of study in a portfolio signal context for at
least three reasons:
\begin{enumerate}
    \item A signal generating process may exhibit AR(1)-type behavior natively
    \item A signal whose properties appear noise-like becomes AR(1) after
        exponential smoothing (a common filtering operation). 
    \item An AR(1) model may stand as a sufficiently good approximation of a
        more complicated process
\end{enumerate}
Focusing attention on standard Gaussian AR(1) processes is in part motivated by
a standard practice of preparing candidate signals to have, unconditionally,
mean zero, standard deviation one, transformed outliers (through Winsorizing,
trimming, or some other process), and a Gaussian-like profile. A second
motivation comes from the common practice of passing a ``fast" signal (one with
autocorrelation much less than 1) through a linear low-pass filter, such as an
infinite impulse response moving average filter.

A third motivation relates to the ergodic behavior of such processes.
Informally speaking, an AR(1) process is positive as much and as often as it is
negative. Translated into a portfolio context, this implies that, over suitably
long horizons, positive or negative positions should not ``build up" (which may
suggest unresolved cointegration effects to remove, or would have adverse risk
implications for a portfolio over the long term).

For our purposes, we restrict attention to Gaussian AR(1) processes with
autocorrelation between zero (the iid Gaussian case) and one. For convenience,
we normalize such processes to have standard deviation one. In fact, every such
AR(1) process may be obtained from an iid Gaussian process through exponential
smoothing, and so we introduce the 1-parameter family of smoothed signals
$\tX_t$, which are obtained from $X_t$ through exponential smoothing:
\begin{defn}[Exponential moving average]\label{defn EMA}
Given $0 \leq \alpha < 1$, define
    \begin{equation}
        X_t \xrightarrow{\text{EMA}} \tX_t = \sqrt{1 - \alpha^2}(X_t + \alpha X_{t-1} + \alpha^2 X_{t-2} + \alpha^3 X_{t-3} + ...)
    \end{equation}
\end{defn}
Note that $\tX_t \sim \mathcal{N}(0,1)$ unconditionally and
\begin{equation} \label{defn:tilde_xt}
    \tX_t = \alpha \tX_{t-1} + \sqrt{1-\alpha^2} X_t.
\end{equation}
Observe that $\tX_{t-1}$ and $X_t$ are independent and that, by construction,
$\tX_t$ is an AR(1) process with autocorrelation $\alpha$.  We note that an
analogous correspondence exists in the continuous case \cite{BiPaTa08}.

\subsection{Constraining the portfolio}
In this section we take the signal $\tX_t = \tX_t(\alpha)$ as given and
and consider a simple portfolio construction procedure based on
thresholding and hysteresis:
\begin{defn}[Portfolio construction rule]
    Given smoothed signal $\tX_t$ and parameter $\eta > 0$,
    we the construct the portfolio $W_t$ by
    \begin{equation} \label{defn:w_t}
        W_t = \left\{\begin{aligned}
            & 1,\hspace{2em} \tX_t \geq \eta \\
            & -1, \hspace{2em} \tX_t \leq -\eta \\
            & 1,\hspace{2em} -\eta < \tX_t < \eta, W_{t-1} = 1 \\
            & -1,\hspace{2em} -\eta < \tX_t < \eta, W_{t-1} = -1 
        \end{aligned}\right.
    \end{equation}
\end{defn}
In the absence of transaction costs, an optimal portfolio would take a position
proportional to the strength of the signal (at least under the assumptions we
have made on our stochastic processes). In practice, this behavior is not
followed in certain regimes. For example, if the absolute value of the signal
is sufficiently large (in theory, it is unbounded), then a proportional
position may exceed a risk budget. On the other hand, small deviations in the
signal may imply small trades, which, in the presence of transaction costs,
may cease to have positive expected value. The first point may be formalized
in a risk model by imposing a limit on the $L^\infty$ norm in a mean-variance
optimization. The latter behavior may be observed should a transaction cost
term, proportional to the absolute value of the trade size, be included in the
optimization problem (e.g., see the discussion of a ``no-trade" region in
\cite{DeMeNo16}).

Rather than consider the full mean-variance optimization problem, we simplify
by admitting two states: a ``long" position of +1, and a ``short" position of
-1. See figure \ref{fig portfolio} for an illustration of how the thresholds
define the portfolio. Here the upper and lower thresholds are set to be +1 and
-1.
\begin{figure}[h]
    \centering
    \includegraphics[width=0.5\textwidth]{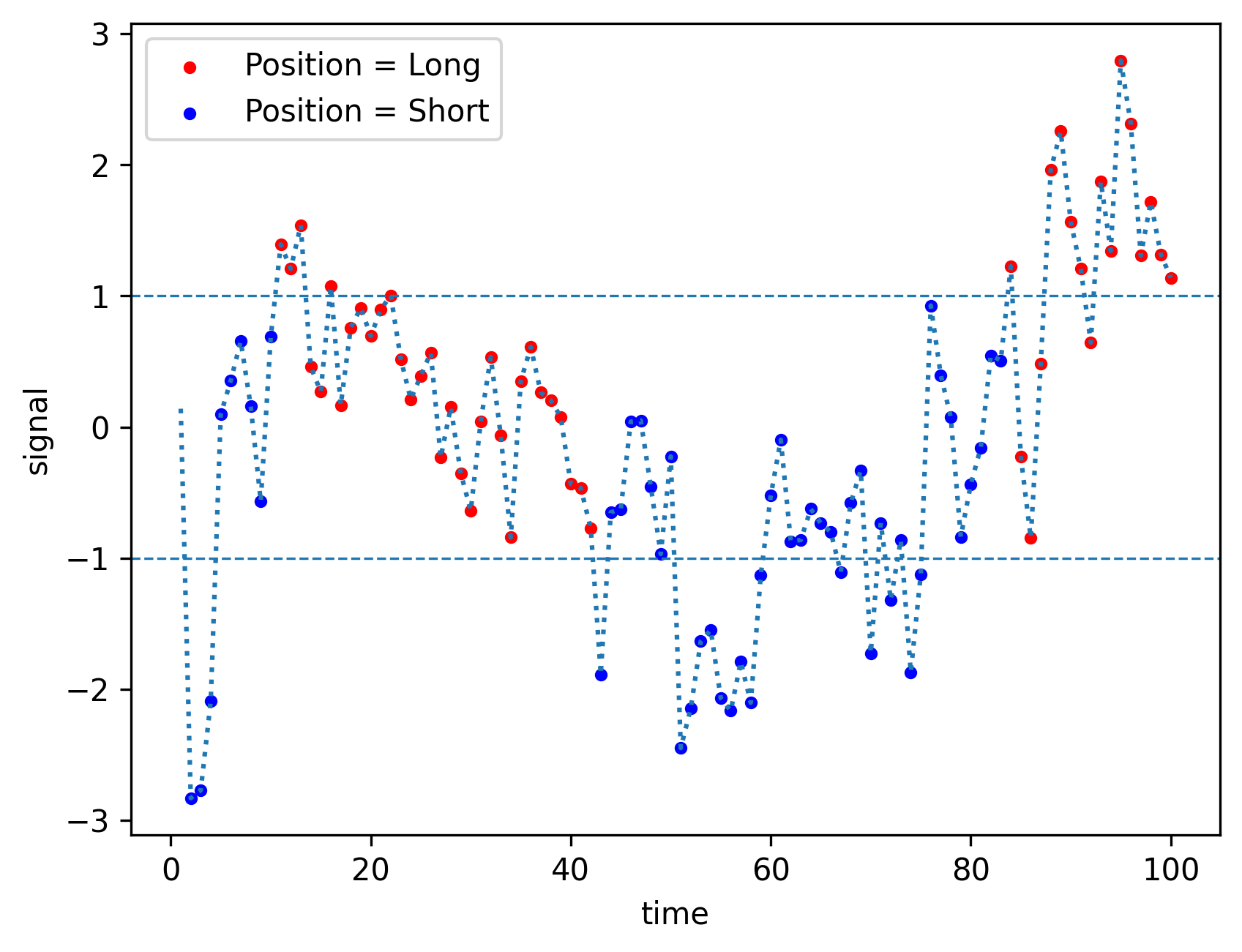}
    \caption{An AR(1) signal and its corresponding portfolio}
    \label{fig portfolio}
\end{figure}

\subsection{The optimization problem}

The return of the portfolio can be measured by the correlation between target
and portfolio.
\begin{defn}[Portfolio target correlation function]
Define $K(\alpha, \eta)$ by
\begin{equation}\label{correlation func}
    K(\alpha, \eta) := \E[W_t Y_t]
\end{equation}
\end{defn}
We acknowledge that this measurement is in fact an approximation of return,
as the value of the portfolio will in general fluctuate from period to
period in the absence of trading. However, when sufficiently short time
periods and typical scales of returns are considered, these effects are
second order, and so for simplicity we choose to ignore them in favor of a
simple measurement of performance.

Various models of transaction costs and the resulting optimization problems
have been studied, e.g. linear model in \cite{davis1990portfolio}, quadratic
model in \cite{garleanu2013dynamic}, and an in-between model in
\cite{boyd2017multi}.  In this paper, we do not assume a specific model for
transaction costs. Instead, we measure the negative effects of transaction
costs by trading frequency. Due to transaction costs and the risk of slippage,
we generally want to reduce the frequency of trading. We define the following
survival time function to measure the time needed for $\tX_t$ to pass the lower
barrier $-\eta$ given some starting position $x_0 \in \R$.

\begin{defn}[Survival time]
    Let $\alpha \in [0, 1)$ and suppose
    $\tX_t$ and $W_t$ are defined as in \eqref{defn:tilde_xt} and
    \eqref{defn:w_t}, respectively.
    For $\eta > 0$ and initial position $\tX_0 = x_0$, we define the survival
    time random variable $\tau_{\alpha, \eta, x_0}$ as
    \begin{equation}\label{stopping_time}
        \tau_{\alpha, \eta, x_0} := \inf_{t} \{t | \tX_t \leq -\eta \}
    \end{equation}
    We define the expected survival time as
    \begin{equation}\label{target function}
        H(\alpha, \eta) = \E_{x_0}[\tau_{\alpha, \eta, x_0} | x_0 > \eta]
    \end{equation}
\end{defn}

We formulate the problem as follows:

\begin{opt}\label{prob1}
    Given $c > 0$, maximize the expected survival time \eqref{target function}
    under the correlation constraint $K(\alpha, \eta) \geq c$.
\end{opt}

\section{Main results}
Our main result is stated in the following Theorem:
\begin{theorem}\label{thm opt}
    A locally optimal solution to the constrained optimization problem
    \ref{prob1} occurs at $\alpha = 0$ and $\eta_0$ such that $K(0, \eta_0) =
    c$. In other words, to optimize the portfolio it suffices to adjust the
    threshold without applying any small amount of AR(1) smoothing.  The
    appendix includes technical results required by our argument.
\end{theorem}
In figure \ref{fig survival time contour}, we plot the contour lines of the
expected survival time function generated from a numerical simulation of the
process.
\begin{figure}[h]
    \centering
    \includegraphics[width=0.5\textwidth]{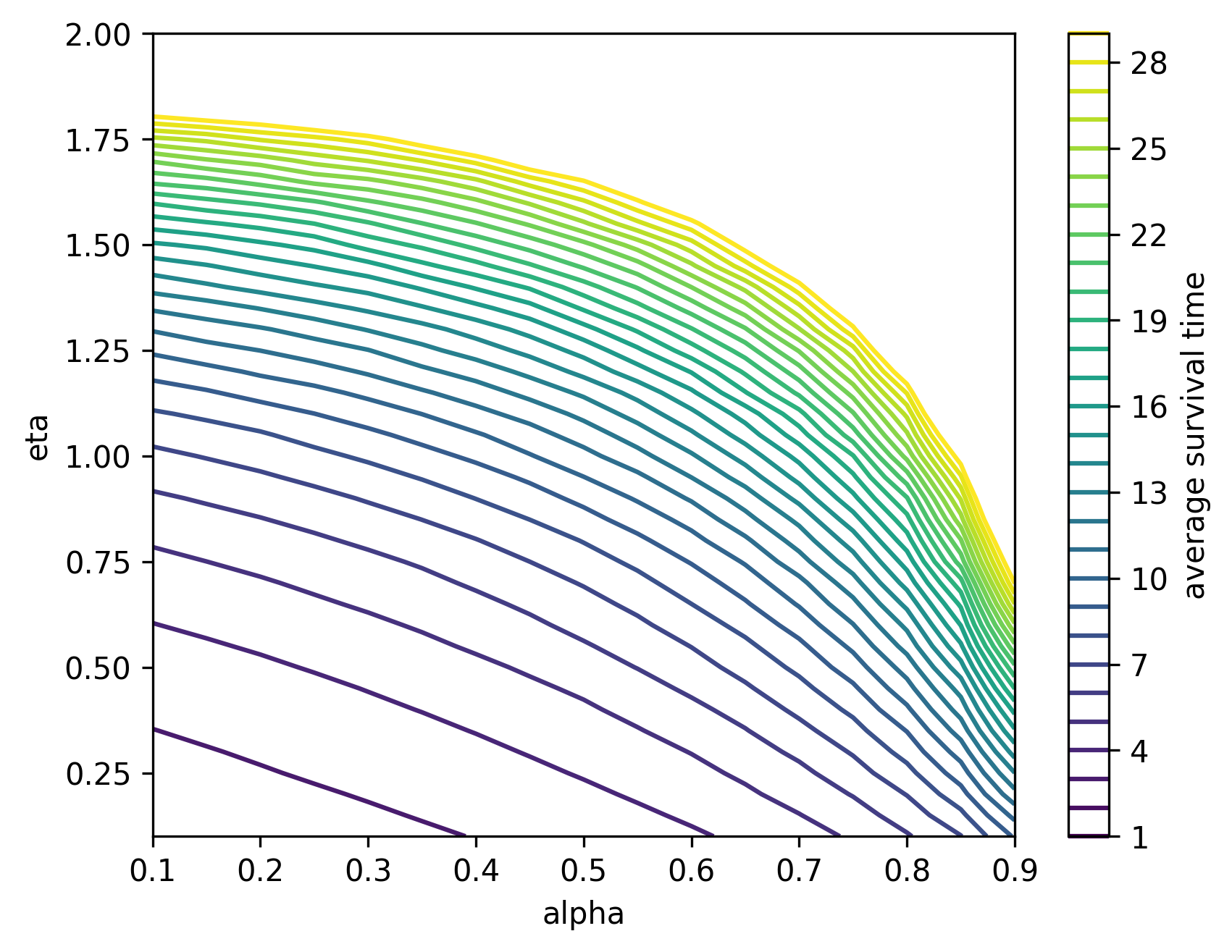}
    \caption{Contour lines of $H(\alpha, \eta)$}
    \label{fig survival time contour}
\end{figure}
We will use the Lagrange multiplier method to prove theorem \ref{thm opt}.
The outline of the paper is as follows.
In \S \ref{sec:corr}, we study the correlation function over the full parameter
space $(\alpha, \eta) \in [0, 1) \times [0, \infty)$ and prove stronger results
near $\alpha = 0$.
In \S \ref{sec:survival_time}, we study the survival time function and prove
the main result in the next section. In \S \ref{sec:future dir}, we comment on
a few related problems remaining to be explored.

\section{Properties of the correlation function}\label{sec:corr}

We begin by decomposing the correlation function \eqref{correlation func} into
two pieces, which we handle separately.

Define $W(x):\R \rightarrow \R$ by
\begin{equation}
    W(x) =
    \begin{cases}
        1 & x \geq \eta \\
        0 & -\eta < x < \eta \\
        -1 & x \leq -\eta
    \end{cases}
\end{equation}
Note that this is distinct from, but related to, the portfolio $W_t$ at time
$t$.

Using \eqref{y_t} and the independence of $W_t$ and $\varepsilon_t$, we have
\begin{equation}\label{eq: E[W_tY_t]}
    \begin{aligned}
        K(\alpha, \eta)
        &= \E(W_tY_t) \\
        &= \E(W_t(\rho X_t + \sqrt{1 - \rho^2} \varepsilon_t)) \\
        &= \rho \E(W_t X_t) \\
    \end{aligned}
\end{equation}
Let $\chi(S)$ be the characteristic function of set $S$ in the sample space. We
have
\begin{equation*}
    1 = \chi(\tX_t \geq \eta) + \chi(\tX_t \leq -\eta) +
        \chi(-\eta < \tX_t < \eta, W_{t - 1} = 1) +
        \chi(-\eta < \tX_t < \eta, W_{t - 1} = -1)
\end{equation*}
Using $\chi$, we may split \eqref{eq: E[W_tY_t]} into
\begin{equation}\label{expected return split eq}
    \begin{aligned}
        (\ref{eq: E[W_tY_t]}) = \;&
        \rho \E(W(\tX_t) X_t \chi(\tX_t \geq \eta)) \\
        &+ \rho \E(W(\tX_t) X_t \chi(\tX_t \leq -\eta)) \\
        & + \rho \E(X_t \chi(-\eta < \tX_t < \eta, W_{t - 1} = 1)) \\
        &- \rho \E(X_t \chi(-\eta < \tX_t < \eta, W_{t - 1} = -1)) \\
    \end{aligned}
\end{equation}
Denote the first two terms in \eqref{expected return split eq} by
$K_0(\alpha, \eta)$, i.e.,
\[
    K_0(\alpha, \eta) =
    \rho \left[
    \E(W(\tX_t) X_t \chi(\tX_t \geq \eta))
    + \E(W(\tX_t) X_t \chi(\tX_t \leq -\eta))
    \right]
\]
and the last two terms by $\cE(\alpha, \eta)$:
\[
    \cE(\alpha, \eta) =
    \rho \left[
    \E(X_t \chi(-\eta < \tX_t < \eta, W_{t - 1} = 1))
    - \E(X_t \chi(-\eta < \tX_t < \eta, W_{t - 1} = -1))
    \right].
\]
We may consider $G_0$ as the contribution from a ``strong" signal (relative to
the threshold $\eta$), and $\cE$ as the contribution from a ``weak" signal,
which, given the thresholding rule, results in a hysteresis effect.

\subsection{Estimating the contribution of a strong signal}

We compute $K_0$ explicity.
\begin{equation}
    \begin{aligned}
        K_0(\alpha, \eta)
        &= \rho \E(W(\alpha \tX_{t-1} + \sqrt{1-\alpha^2} X_t) X_t) \\
        &= 2\rho \iint_{\sqrt{1 - \alpha^2} x + \alpha z \geq \eta} x
           \cdot \frac{1}{2\pi}e^{-\frac{x^2 + z^2}{2}} dxdz \\
        &= \frac{2\rho}{2\pi} \int_{-\infty}^{+\infty}
           \int_{-\frac{\alpha}{\sqrt{1 - \alpha^2}} z +
           \frac{\eta}{\sqrt{1 - \alpha^2}}}^{+\infty} xe^{-\frac{x^2}{2}}dx
           e^{-\frac{z^2}{2}} dz \\
        &= \frac{2\rho}{2\pi} \int_{-\infty}^{+\infty}
           e^{-\frac{1}{2(1 - \alpha^2)}
           (\alpha^2 z^2 - 2\alpha\eta z + \eta^2 )}
           \cdot e^{-\frac{z^2}{2}} dz \\
        &= \frac{2\rho}{\sqrt{2\pi}} \int_{-\infty}^{+\infty}
           \frac{1}{\sqrt{2\pi(1 - \alpha^2)}} e^{-\frac{1}{2(1 - \alpha^2)}
           (z - \alpha \eta)^2 - \frac{\eta^2}{2}} dz \\
        &= \frac{2\rho}{\sqrt{2\pi}}e^{-\frac{\eta^2}{2}}\sqrt{1 - \alpha^2}.
    \end{aligned}
\end{equation}

\subsection{Estimating the drag from hysteresis}

The two terms in $\cE$ are both negative and equal due to symmetry.
We estimate them by further conditioning on $W_k$, $ k \leq t$.
Let $\Omega_k$, $k=0, 1, 2, \ldots,$ denote the event
\begin{equation}
    \Omega_k = \bigcap_{j=0}^{k} \{ -\eta < \tX_{t-j} < \eta \}.
\end{equation}
Then we can decompose $\cE$ as the series
\begin{equation}\label{error term sum}
    \begin{aligned}
        \cE(\alpha, \eta)
         = & 2\rho \E(X_t \chi(-\eta < \tX_t < \eta, W_{t - 1} = 1)) \\
         = & 2\rho \E(X_t \chi(-\eta < \tX_t < \eta, \tX_{t - 1} \geq \eta))
           + 2\rho \E(X_t \chi(\Omega_1, W_{t - 2} = 1)) \\
         = & 2\rho \E(X_t \chi(\Omega_0, \tX_{t - 1} \geq \eta)) +
             2\rho \E(X_t \chi(\Omega_1, \tX_{t - 1} \geq \eta)) \\
        & + 2\rho \E(X_t \chi(\Omega_2, W_{t - 3} = 1)) \\
        = & \cdots \\
        = & \sum_{k=0}^{+\infty} 2\rho \E[X_t \chi(\Omega_k, \tX_{t-k-1}
          \geq \eta)]
    \end{aligned}
\end{equation}
We denote the $k$-th term in the sum by $\cE_k$, i.e.
\begin{equation}
    \cE_k(\alpha, \eta) = 2\rho \E[X_t \chi(\Omega_k, \tX_{t-k-1} \geq \eta)]
\end{equation}
Now we evaluate the $\cE_k$ and their derivatives at $\alpha = 0$.

\begin{lem}\label{lem K at alpha=0}
All $\cE_k$ are non-positive and $\frac{\partial^j \cE_k}{\partial \alpha^j}$
vanish at $\alpha = 0$ for all $0 \leq j \leq k$. In particular,
$\cE_k (0,\eta) = 0$ for all k, and
\begin{equation}\label{k>0 term}
    \begin{aligned}
        & \frac{\partial \cE_k}{\partial \alpha}(0,\eta) = 0 \hspace{2em} for\ k \geq 1\\
        & \frac{\partial^2 \cE_k}{\partial \alpha^2}(0,\eta) = 0 \hspace{2em} for\ k \geq 2
    \end{aligned}
\end{equation}

We also compute the following quantities explicitly:
\begin{equation}\label{k=0 term}
    \cE_0(\alpha, \eta) =
    \frac{2\rho}{\sqrt{2\pi}} e^{-\frac{\eta^2}{2}} \sqrt{1 - \alpha^2}
    \left(F(\sqrt{\frac{1 - \alpha}{1 + \alpha}}\eta)
    - F(\sqrt{\frac{1 + \alpha}{1 - \alpha}}\eta) \right)
\end{equation}
and
\begin{equation}
    \frac{\partial^2 \cE_1}{\partial \alpha^2}(0, \eta) =
    -8\rho \eta f(\eta)^2 (F(\eta) - F(-\eta) - 2f(\eta))
\end{equation}
where $F(x)$ and $f(x)$ are respectively the cdf and pdf of a standard Gaussian
distribution.

\end{lem}
\begin{proof}
We first prove \eqref{k=0 term}. Recall that
\begin{equation*}
    \tX_t = \alpha \tX_{t-1} + \sqrt{1 - \alpha^2} X_t
\end{equation*}
and $X_t, \tX_{t-1}$ are independent.
The LHS of \eqref{k=0 term} is equal to
\begin{equation}
    \begin{aligned}
        & 2\rho \int_\eta^{+\infty} \int_{-\frac{\alpha y}{\sqrt{1 - \alpha^2}} -
        \frac{\eta}{\sqrt{1 - \alpha^2}}}^{-\frac{\alpha y}{\sqrt{1 - \alpha^2}}
        + \frac{\eta}{\sqrt{1 - \alpha^2}}} x \cdot
        \frac{1}{2\pi}e^{-\frac{x^2 + y^2}{2}} dxdy
    \end{aligned}
\end{equation}
and \eqref{k=0 term} follows from direct computation.
We now prove \eqref{k>0 term} by showing that
\begin{equation}\label{error term 1 in G}
    \cE_k \leq \alpha^{k+1} C(k, \eta)
\end{equation}
for some $C(k, \eta) > 0$ and for $\alpha$ bounded away from 1.
This indicates that the first order partial derivative in $\alpha$ is zero at
$\alpha = 0$. We prove the claim for $k=1$. The proof for larger $k$ follows
from similar arguments.

Express the LHS of \eqref{error term 1 in G} by an iterated integral, where
variables for integration and the random variables correspond as below:
\begin{equation*}
        x \sim X_t ,\
        y \sim X_{t-1} ,\
        z \sim \tX_{t-2}.
\end{equation*}
The integral representation of \eqref{error term 1 in G}
\begin{equation}\label{eq triple int}
    \iiint_\Sigma x \cdot \frac{1}{(\sqrt{2\pi})^3}
    e^{-\frac{x^2 + y^2 + z^2}{2}} dxdydz
    = \iint_{(y,z) \in \Sigma_0} \int_{x \in I(y,z)} xe^{-\frac{x^2}{2}} dx
    \frac{1}{(\sqrt{2\pi})^3} e^{-\frac{y^2 + z^2}{2}} dydz,
\end{equation}
where the regions are
\begin{equation}
    \begin{aligned}
        & \Sigma_0 = \{ (y,z) | z \geq \eta, -\eta \leq \sqrt{1 - \alpha^2}y + \alpha z \leq \eta \} \\
        & I(y,z) = \{-\eta \leq \sqrt{1 - \alpha^2}x + \alpha(\sqrt{1 - \alpha^2}y + \alpha z) \leq \eta\} \\
        & \Sigma = \{(x,y,z) | (y,z) \in \Sigma_0, x \in I(y,z)\}.
    \end{aligned}
\end{equation}
Let us examine the inner integral w.r.t $x$.
The interval $I(y,z)$ is equal to a shifted version of
\begin{equation*}
    I_0 = (-\frac{\eta}{\sqrt{1 - \alpha^2}}, \frac{\eta}{\sqrt{1 - \alpha^2}})
\end{equation*}
which is centered at origin. The shift is equal to
\begin{equation*}
    -\frac{\alpha}{\sqrt{1 - \alpha^2}}\eta \leq
    \frac{1}{\sqrt{1 - \alpha^2}}(\alpha(\sqrt{1 - \alpha^2}y + \alpha z))
    \leq \frac{\alpha}{\sqrt{1 - \alpha^2}}\eta.
\end{equation*}
Note that the $x$-integral on the unshifted interval $I_0$ is equal to 0.
When $\alpha$ is small, cancelation is still valid on a large portion of $I$,
leaving only a subinterval of length $\mathcal{O}(\alpha)$ that contributes to
the integral. \\
Thus, we have
\begin{equation}\label{cancel along x}
    \int_{x \in I(y,z)} xe^{-\frac{x^2}{2}} dx \leq C(\eta) \alpha
\end{equation}
uniformly in $y, z$.
The other $\alpha$ factor results from cancelation in the $y$-$z$ plane.
In fact, if we make the change of variable
\begin{equation*}
    \begin{aligned}
        & u = \sqrt{1 - \alpha^2}y + \alpha z \\
        & v = -\alpha y + \sqrt{1 - \alpha^2} z
    \end{aligned}
\end{equation*}
and abuse the notation by denoting $I(y(u,v), z(u,v))$ by $I(u,v)$ for
simplicity, then we have
\begin{equation}\label{cancel along u}
    \int_{I(u, v)} x e^{-\frac{x^2}{2}}dx =
    -\int_{I(-u, v)} x e^{-\frac{x^2}{2}} dx.
\end{equation}
Also, note that $y^2 + z^2 = u^2 + v^2$. Thus
\begin{equation}
    \frac{1}{(\sqrt{2\pi})^3} e^{-\frac{y^2 + z^2}{2}} =
    \frac{1}{(\sqrt{2\pi})^3} e^{-\frac{u^2 + v^2}{2}}
\end{equation}
is even in $u$ for each fixed $v$.

\begin{figure}
    \centering
    \begin{tikzpicture}
        \begin{axis}[
            axis lines = middle,
            xlabel = \( y \),
            ylabel = \( z \),
            domain = -5:2,
            xmin = -5, xmax = 2,
            ymin = 0, ymax = 6,
            xticklabel=\empty,
            yticklabel=\empty,
            legend style = {at={(1,1.05)}, anchor=south east},
        ]
            \addplot[
                domain=-5:1,
                samples=100,
                thick,
                color=black
            ] {2-x}; 
            \addlegendentry{\(l_1: \sqrt{1-\alpha^2} y + \alpha z = \eta\)}
            \addplot[
                domain=-5:-3,
                samples=100,
                thick,
                color=black
            ] {-2-x}; 
            \addlegendentry{\(l_2: \sqrt{1-\alpha^2} y + \alpha z = -\eta\)}
            \addplot[
                domain=-3:1,
                samples=100,
                thick,
                color=black
            ] {1}; 
            \addplot[
                domain=-3:-1,
                samples=100,
                thick,
                color=black,
            ] {x+4};
            \node[] at (axis cs: -3,3) {$\Sigma_{0,1}$};
            \node[] at (axis cs: -1,1.6) {$\Sigma_{0,2}$};
            \node[] at (axis cs: -4,5.5) {$l_1$};
            \node[] at (axis cs: -4.5,2) {$l_2$};
        \end{axis}
        \draw[->, thick, blue] (-3,2) -- (-2,3) node[midway, above] {$\mathbf{u}$};
        \draw[->, thick, blue] (-3,2) -- (-4,3) node[midway, below] {$\mathbf{v}$};
    \end{tikzpicture}
    \caption{Plot of \( \Sigma_{0,1} \) and \( \Sigma_{0,2} \).}
\end{figure}
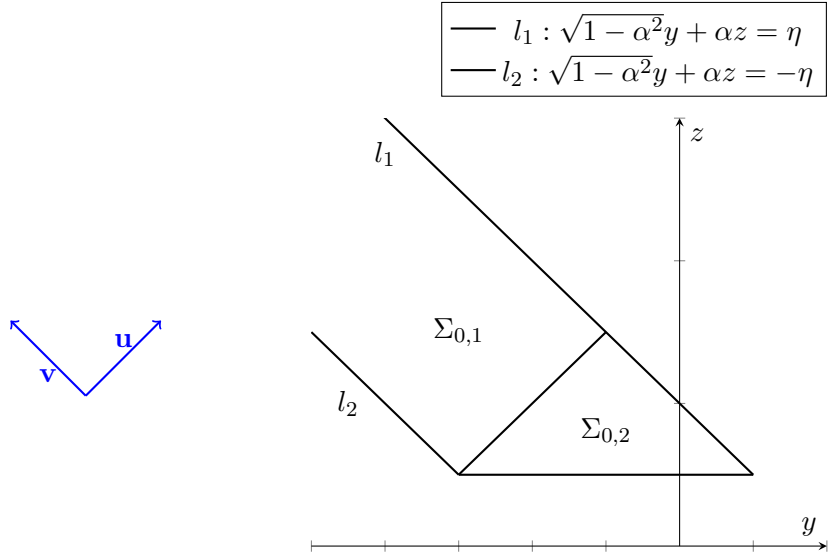
    
With the new variables, the region of the double integral becomes
\begin{equation*}
    \begin{aligned}
        \Sigma_0 = & \{ (u, v) | -\eta < u < \eta,\ v > \sqrt{\frac{1 + \alpha}{1 - \alpha}} \eta \} \\
        \cup & \{ -\eta < u < \eta,\ \sqrt{\frac{1 + \alpha}{1 - \alpha}}\eta -
        (u + \eta) \frac{\alpha}{\sqrt{1 - \alpha^2}} < v <
        \sqrt{\frac{1 + \alpha}{1 - \alpha}}\eta \} \\
        := & \Sigma_{0,1} + \Sigma_{0,2}
    \end{aligned}
\end{equation*}
where $\Sigma_{0,1}, \Sigma_{0,2}$ are rectangular and triangular regions,
respectively.
The line integral of $\int_{I(u,v)} x e^{-\frac{x^2}{2}} dx$ along the $u$
direction cancels in $\Sigma_{0,1}$ due to
\eqref{cancel along u}: for $v \geq \sqrt{\frac{1+\alpha}{1-\alpha}}\eta$,
\begin{equation}
    \int_{-\eta}^{+\eta}\int_{I(u, v)} x e^{-\frac{x^2}{2}}dx du = 0
\end{equation}

The area of $\Sigma_{0,2}$ is bounded by $C(\eta)\alpha$.
Pairing this fact with \eqref{cancel along x}, we get
\begin{equation}
    \iint_{\Sigma_{0,2}} \int_{I(u,v)} x \cdot \frac{1}{(\sqrt{2\pi})^3}
    e^{-\frac{x^2+u^2+v^2}{2}} dx dudv \leq C(\eta)\alpha^2
\end{equation}
This proves the $\alpha^2$ bound. To show the negativity, one only need to
notice that
\begin{equation}
    \iint_{\Sigma_{0,2}}\int_{I(u, v)} x e^{-\frac{x^2}{2}}dx du \leq 0
\end{equation}
due to missing contribution from negative $u$. This concludes the proof for
the $k=1$ term. The argument for the higher order terms is similar.

We need the second order derivatives of $K(\alpha, \eta)$. The
$\frac{\partial^2}{\partial \eta^2}$ and
$\frac{\partial^2}{\partial \alpha \partial \eta}$ ones are relatively easy.
Here we compute the $\frac{\partial^2}{\partial \alpha^2}$. The $K_0$ term and
the $k=0$ term of $\cE$ have explicit expressions. Note that $k\geq 2$
terms of $\cE$ do not contribute to
$\frac{\partial^2 \cE}{\partial \alpha^2}$. It only remains to compute that of
the $k=1$ term, i.e. \eqref{eq triple int}. Note that the integral region for
$x$ (after $u$, $v$ change of coordinates) is
\begin{equation}
    I(u,v) = [-\frac{\alpha}{\sqrt{1-\alpha^2}}u +
    \frac{\eta}{\sqrt{1-\alpha^2}} ,
    \frac{\alpha}{\sqrt{1-\alpha^2}}u + \frac{\eta}{\sqrt{1-\alpha^2}}]
\end{equation}
Integrating in $x$ and recalling that $\Sigma_{0,1}$ does not contribute to the
triple integral, we have
\begin{equation}
\eqref{eq triple int} = \int_{\Sigma_{0,2}} \frac{1}{2\pi}
    e^{-\frac{u^2+v^2}{2}} [f(\frac{\alpha}{\sqrt{1-\alpha^2}}u +
    \frac{\eta}{\sqrt{1-\alpha^2}}) - f(-\frac{\alpha}{\sqrt{1-\alpha^2}}u +
    \frac{\eta}{\sqrt{1-\alpha^2}})] du dv
\end{equation}
We do a Taylor series expansion to the integrand w.r.t. $\alpha$ around
$\alpha=0$ and get
\begin{equation}
    \eqref{eq triple int} = \int_{\Sigma_{0,2}} \frac{1}{2\pi}
    e^{-\frac{u^2+v^2}{2}}f'(\eta)\cdot 2\alpha u + H.O.T dudv
\end{equation}

Rewrite the bounds of $\Sigma_{0,2}$ as
\begin{equation}
    \sqrt{\frac{1+\alpha}{1-\alpha}}\eta -
    \frac{2\alpha\eta}{\sqrt{1-\alpha^2}} \leq v \leq
    \sqrt{\frac{1+\alpha}{1-\alpha}}\eta, \hspace{2em}
    -\frac{\sqrt{1-\alpha^2}}{\alpha}v + \frac{\eta}{\alpha} \leq u \leq \eta
\end{equation}
and make another change of variables
\begin{equation}
    v =\sqrt{\frac{1+\alpha}{1-\alpha}}\eta +
    \frac{\alpha\eta}{\sqrt{1-\alpha^2}} w, \hspace{2em} -2 \leq w \leq 0
\end{equation}

Under the new variables, $\Sigma_{0,2}$ becomes
\begin{equation}
    -2 \leq w \leq 0, \hspace{2em} -(w+1)\eta \leq u \leq \eta
\end{equation}
Then we have
\begin{equation}
    \begin{aligned}
        \eqref{eq triple int} & = 2\alpha f'(-\eta) \int_{-2}^0
        \int_{-(w+1)\eta}^\eta \frac{1}{\sqrt{2\pi}} u e^{-\frac{u^2}{2}} du
        \frac{1}{\sqrt{2\pi}}
        e^{-\frac{1}{2}(\sqrt{\frac{1+\alpha}{1-\alpha}}\eta +
        \frac{\alpha\eta}{\sqrt{1-\alpha^2}}w)^2}
        \frac{\alpha\eta}{\sqrt{1-\alpha^2}}dw + H.O.T \\
    \end{aligned}
\end{equation}
The coefficient of $\alpha^2$ evaluated at $\alpha=0$ gives us (one half) the
second order derivative. Thus
$\frac{\partial^2 \cE_1}{\partial \alpha^2}(0,\eta)$ is equal to
\begin{equation}
    8 \eta f'(\eta) f(\eta) \int_{-2}^0 (f(-(w+1)\eta) - f(\eta)) dw =
    -8 \eta f(\eta)^2 (F(\eta) - F(-\eta) - 2f(\eta))
\end{equation}

\end{proof}

In the next lemma, we compute the gradient of K at $\alpha = 0$.
\begin{lem}
    $\nabla K(0, \eta)$ is given by
    \begin{equation}
        \begin{aligned}
            & \frac{\partial K}{\partial \alpha}(0, \eta) =
            -\frac{2\rho}{\pi} \eta e^{-\eta^2} \\
            & \frac{\partial K}{\partial \eta}(0, \eta) =
            -\frac{2\rho}{\sqrt{2\pi}}\eta e^{-\frac{\eta^2}{2}}
        \end{aligned}
    \end{equation}
\end{lem}
\begin{proof}
Since the $\cE_k\ (k \geq 1)$ do not contribute to
$\frac{\partial K}{\partial \alpha}(0, \eta)$, we have
    \begin{equation*}
        \begin{aligned}
            \frac{\partial K}{\partial \alpha}(0,\eta)
            & = \left.\frac{\partial}{\partial \alpha}\right|_{\alpha = 0}
            K_0(\alpha, \eta) + \frac{2}{\sqrt{2\pi}}e^{-\frac{\eta^2}{2}}
            \sqrt{1 - \alpha^2}
            \left( F(\sqrt{\frac{1 - \alpha}{1 + \alpha}}\eta) -
            F(\sqrt{\frac{1 + \alpha}{1 - \alpha}}\eta) \right) \\
            & = \left.\frac{\partial}{\partial \alpha}\right|_{\alpha = 0}
            \frac{2}{\sqrt{2\pi}}e^{-\frac{\eta^2}{2}} \sqrt{1 - \alpha^2}
            \left( F(\sqrt{\frac{1 - \alpha}{1 + \alpha}}\eta) -
            F(\sqrt{\frac{1 + \alpha}{1 - \alpha}}\eta) \right) \\
            = & -\frac{2\rho}{\pi} \eta e^{-\eta^2}
        \end{aligned}
    \end{equation*}
    For the $\eta$ direction, we have
    \begin{equation*}
        \frac{\partial K}{\partial \eta}(0, \eta) =
        -\frac{2\rho}{\sqrt{2\pi}}\eta e^{-\frac{\eta^2}{2}}
    \end{equation*}
    This comes from the explicit expression
    \begin{equation*}
        K(0,\eta) = K_0(0,\eta) = \frac{2\rho}{\sqrt{2\pi}} e^{-\frac{\eta^2}{2}}
\end{equation*}

\end{proof}

At the end of this subsection, we compute the second order derivatives of K at
$\alpha=0$. We compute $\frac{\partial^2 K}{\partial \alpha^2}(0,\eta)$ first.
In view of the preceding lemma, it only remains to compute
$\frac{\partial^2 K_0}{\partial \alpha^2}(0,\eta)$ and
$\frac{\partial^2 \cE_0}{\partial \alpha^2}(0,\eta)$.
\begin{equation}
    \frac{\partial^2 K_0}{\partial \alpha^2}(0,\eta) =
    - \frac{2}{\sqrt{2\pi}} e^{-\frac{\eta^2}{2}}
\end{equation}
and
\begin{equation}
    \frac{\partial^2 \cE_0}{\partial \alpha^2}(0,\eta) = 0
\end{equation}
Thus
\begin{equation}
    \frac{\partial^2 K}{\partial \alpha^2}(0,\eta) =
    -2\rho f(\eta) - 8\rho \eta f(\eta)^2 (F(\eta) - F(-\eta) - 2f(\eta))
\end{equation}
The remaining second order derivatives of K are calculated below.
\begin{equation}
    \frac{\partial^2 K}{\partial \alpha \partial \eta}(0,\eta) =
    \frac{2\rho}{\pi}(2\eta^2 - 1)e^{-\eta^2}
\end{equation}
and
\begin{equation}
    \frac{\partial^2 K}{\partial \eta^2}(0,\eta) =
    \frac{2\rho}{\sqrt{2\pi}}(\eta^2 - 1)e^{-\frac{\eta^2}{2}}
\end{equation}

\section{Expected survival times}\label{sec:survival_time}

In this section, we define the function which measures the expected survival
time for an AR(1) process to cross the lower barrier $-\eta$ when starting
from above $+\eta$.
The survival time random variable $\tau_{\alpha, \eta, x}$ is defined as in
\eqref{stopping_time}.
We define $h(x, \alpha, \eta)$ to be the expectation of
$\tau_{\alpha, \eta, x}$:
\begin{equation}
    h(x, \alpha, \eta) := \E[\tau_{\alpha, \eta, x} | x].
\end{equation}

\subsection{The expected stopping time $h(x, \alpha, \eta)$}
\begin{rem}
    The expected stopping time $h(x, \alpha, \eta)$ satisfies the following
    integral equation:
    \begin{equation}\label{integral eq for survival time}
        \begin{aligned}
            & h(x, \alpha, \eta) = 1 + \int_{-\eta}^{+\infty}
              h(\alpha, \eta, y)f(\frac{y - \alpha x}{\sqrt{1 - \alpha^2}})
              \frac{dy}{\sqrt{1 - \alpha^2}} \ for\ x>-\eta \\
            & h(x, \alpha, \eta) = 0\ for\ x\leq -\eta
        \end{aligned}
    \end{equation}
\end{rem}
where, recall, $F(x)$ and $f(x)$ are the cdf and pdf of a standard Gaussian
distribution, respectively. To see why \eqref{integral eq for survival time}
holds, first note that the transition probability of $\tX_t$ is
\begin{equation}
    \begin{aligned}
        P(\tX_1 \in dy | \tX_0 = x)
        &= P(\alpha x + \sqrt{1 - \alpha^2} X_1 \in dy) \\
        &= f(\frac{y - \alpha x}{\sqrt{1 - \alpha^2}})
           \frac{dy}{\sqrt{1 - \alpha^2}},
    \end{aligned}
\end{equation}
where $f$ is the pdf of $\cN(0,1)$, then use the Markov property.

Equation \eqref{integral eq for survival time} has, locally,
a unique solution, and this solution belongs to the class $C^1$:
\begin{lem}\label{thm pde solution}
    For fixed $\eta>0$, let
    $D = \{(x, \alpha) | -\eta < x < +\infty, 0 \leq \alpha < 1, \eta \geq 0 \}$
    \begin{enumerate}
        \item  There exists a unique $h(x, \alpha, \eta)$ defined on D which
               solves \eqref{integral eq for survival time}
        \item  The solution $h$ satisfies the following logarithmic growth
               bound in $x$: for some constants $A_1(\alpha, \eta) > 0$ and
               $A_2(\alpha, \eta) > 0$ depending upon $\alpha$ and $\eta$,
               we have
        \begin{equation}
            |h(x, \alpha, \eta)| \leq
            A_1(\alpha, \eta) + A_2(\alpha, \eta) \log x
        \end{equation}
        \item  The solution $h$ is locally $C^2$ in D:
            \[
                h \in C^2_{\loc}(D)
            \]
    \end{enumerate}
\end{lem}

The technical details of the proof of lemma \ref{thm pde solution} may be
found in Appendix \ref{AppendixA}. Here we note that a straightforward
application of a contraction mapping argument fails due to the unboundedness of
the expected survival time as $x$ approaches infinity. To overcome this, we
define a family of approximating operators to which a contraction mapping
argument may be established via a priori bounds. We show that the a priori
bounds are uniform with respect to the approximation, which then allows us to
use the Arzel\'a-Ascoli theorem and a standard diagonal argument to extract a
uniformly convergent subsequence.

\subsection{The expected survival time $H(\alpha, \eta)$}

We now express the expected survival time $H(\alpha, \eta)$, i.e., the expected
number of steps it takes to cross the lower barrier $-\eta$ conditioned on
having started above the upper barrier $\eta$ (see \eqref{target function}), in
terms of $h$:
\begin{equation}
    \begin{aligned}
        H(\alpha, \eta)
        &= \frac{1}{1 - F(\eta)}\int_{\eta}^{+\infty} h(x, \alpha, \eta) f(x) dx
    \end{aligned}
\end{equation}

In preparation for solving Problem \ref{prob1} using Lagrange multipliers,
we compute the gradient of $H(\alpha, \eta)$ at $\alpha = 0$.

The partial derivatives of $H(\alpha, \eta)$ are given by
\begin{equation}\label{partial E partial alpha eq}
    \frac{\partial H}{\partial \alpha} = \frac{1}{1 - F(\eta)}
    \int_{\eta}^{+\infty} \frac{\partial h}{\partial \alpha}(\alpha, \eta, x)
    f(x) dx
\end{equation}
and
\begin{equation}\label{partial E partial eta eq}
    \begin{aligned}
        \frac{\partial H}{\partial \eta} &=
        \frac{f(\eta)}{(1 - F(\eta))^2} \int_{\eta}^{+\infty} h(x, \alpha, \eta) f(x) dx \\
        &- \frac{1}{1 - F(\eta)} h(\alpha, \eta, \eta+0) f(\eta) \\
        &+ \frac{1}{1 - F(\eta)}\int_{\eta}^{+\infty}
        \frac{\partial h}{\partial \eta}(\alpha, \eta, x)f(x) dx.
    \end{aligned}
\end{equation}
We evaluate \eqref{partial E partial alpha eq} and
\eqref{partial E partial eta eq} on the axis $\alpha = 0$.
\begin{lem}
The expected stopping time $h$ has the following partial derivatives
at $\alpha = 0$:
\begin{equation}
    \begin{aligned}\label{grad_h}
        & h(0, \eta, x) = \frac{1}{F(-\eta)} \\
        & \frac{\partial h}{\partial \alpha}(0, \eta, x)
        = \frac{f(-\eta)^2}{F(-\eta)^2} + \frac{f(-\eta)}{F(-\eta)}x \\
        & \frac{\partial h}{\partial \eta}(0, \eta, x)
        = \frac{f(-\eta)}{F(-\eta)^2}.
    \end{aligned}
\end{equation}
\end{lem}
For the calculations, see Appendix \ref{AppendixB}.
We now compute the gradient of H at $\alpha = 0$.
\begin{lem}\label{lem H partial alpha=0}
$\nabla H$ at $\alpha = 0$ is given by
\begin{equation}
    \begin{aligned}
        \frac{\partial H}{\partial \alpha}(0, \eta) &=
        2 \frac{f(-\eta)^2}{F(-\eta)^2} \\
        \frac{\partial H}{\partial \eta}(0, \eta) &=
        \frac{f(-\eta)}{F(-\eta)^2}.
    \end{aligned}
\end{equation}
\end{lem}
\begin{proof}
Using \eqref{grad_h} in \eqref{partial E partial alpha eq}
and \eqref{partial E partial eta eq}, we obtain
\begin{equation}\label{eq partial H partial alpha}
    \begin{aligned}
        \frac{\partial H}{\partial \alpha}(0, \eta) & =
        \frac{1}{1 - F(\eta)} \int_{\eta}^{+\infty}
        \frac{\partial h}{\partial \alpha} (0, \eta, x) f(x) dx \\
        & = \frac{1}{1 - F(\eta)} \int_{\eta}^{+\infty}
        \left( \frac{f(-\eta)^2}{F(-\eta)^2} + \frac{f(-\eta)}{F(-\eta)}x \right) f(x) dx \\
        & = 2\frac{f(-\eta)^2}{F(-\eta)^2}
    \end{aligned}
\end{equation}
and 
\begin{equation}\label{eq partial H partial eta}
    \begin{aligned}
        \frac{\partial H}{\partial \eta}(0, \eta) & =
        \frac{f(\eta)}{(1 - F(\eta))^2} \int_{\eta}^{+\infty} h(x, 0, \eta) f(x) dx \\
        &- \frac{1}{1 - F(\eta)} h(\eta+0, 0, \eta) f(\eta) \\
        &+ \frac{1}{1 - F(\eta)}\int_{\eta}^{+\infty}
        \frac{\partial h}{\partial \eta}(x, 0, \eta)f(x) dx \\
        & = \frac{f(\eta)}{(1 - F(\eta))^2} \frac{1}{F(-\eta)} (1 - F(\eta)) -
        \frac{1}{1 - F(\eta)} \frac{1}{F(-\eta)} f(\eta) \\
        & + \frac{1}{1 - F(\eta)} \frac{f(-\eta)}{F(-\eta)^2} (1 - F(\eta)) \\
        & = \frac{f(-\eta)}{F(-\eta)^2}.
    \end{aligned}
\end{equation}
\end{proof}

At the end of the subsection, we compute the second order derivatives of $H$,
which are required for Hessian test in the next section.
\begin{lem}\label{lem hessian H}
    The Hessian matrix of H at $\alpha = 0$ is given by
    \begin{enumerate}
        \item $\frac{\partial^2 H}{\partial \alpha^2}(0,\eta) =
            \frac{4f(\eta)}{1 - F(\eta)}\left[ \frac{f(-\eta)^3}{F(-\eta)^2} -
            \frac{\eta f(-\eta)^2}{F(-\eta)} + \frac{f(-\eta)}{F(-\eta)} - f(-\eta) \right] $
        \item $\frac{\partial^2 H}{\partial \alpha \partial \eta}(0,\eta) =
            \frac{4f(-\eta)^2}{F(-\eta)^3} (-\eta F(-\eta) + f(-\eta)) $
        \item $\frac{\partial^2 H}{\partial \eta^2}(0,\eta) =
            \frac{f(-\eta)}{F(-\eta)^3} ( -\eta F(-\eta) + 2f(-\eta) ) $
    \end{enumerate}
\end{lem}

\begin{proof}
    (2) and (3) follow from differentiating \eqref{eq partial H partial alpha}
    and \eqref{eq partial H partial eta} respectively. To compute (1), note
    that
    \begin{equation*}
        \frac{\partial^2 H}{\partial \alpha^2}(0,\eta) =
        \frac{1}{1-F(\eta)} \int_{\eta}^{+\infty}
        \frac{\partial^2 h}{\partial \alpha^2}(0,\eta,x) f(x) dx
    \end{equation*}
    Using Lemma \ref{lem derivatives of h}, we obtain (1).
\end{proof}

\section{Solution to the optimization problem}

In this section we utilize the results of the previous sections to solve
Optimization Problem \ref{prob1} with Lagrange multipliers.

We recall the constraint on the correlation function $K(\alpha ,\eta)$ and
impose the level constraint
\begin{equation}
    K(\alpha, \eta) = c.
\end{equation}
The Lagrange multiplier theorem indicates that a constraint maximum is
achieved at a critical point, where
\begin{equation}\label{eq lagrange multiplier}
    \nabla_{\alpha, \eta} H = \lambda \nabla_{\alpha, \eta} K
\end{equation}
for some $\lambda$.

Thus
\begin{equation}
    \frac{\partial K / \partial \eta}{\partial K / \partial \alpha}(0, \eta)
    = \sqrt{\frac{\pi}{2}} e^{\frac{\eta^2}{2}}
\end{equation}

Using the results of preceding sections, we get
\begin{equation}
    \nabla H(0,\eta) = -\frac{1}{2\rho \eta F(-\eta)^2} \nabla K(0,\eta)
\end{equation}
This shows that Lagrange multiplier equation holds at $\alpha = 0$ with
$\lambda = -\frac{1}{2\rho \eta F(-\eta)^2}$.
To ensure this is a constraint local maximum, we perform the following Hessian
test.
Denote the tangent vector of the level curve $K(\alpha, \eta) = c$ at
$\alpha = 0$ by
$v(\eta)=(\sqrt{\frac{\pi}{2}} e^{\frac{\eta^2}{2}} , -1) = (\frac{1}{2f(\eta)} , -1)$.
Suppose the level curve $K(\alpha, \eta) = c$ is parametrized by
$(\alpha(t), \eta(t))$ with $(\alpha(0), \eta(0)) = (0, \eta)$.
$H$ constrained on $K=c$ obtains maximum at $\alpha = 0$ if
\begin{equation}
    \begin{aligned}
        & \left.\frac{\partial^2 H(\alpha(t), \eta(t))}{\partial t^2}
        \right|_{t=0} < 0 \\
        = & \frac{\partial^2 H}{\partial \alpha^2} \alpha'(0)^2 + 2
        \frac{\partial^2 H}{\partial \alpha \partial \eta} \alpha'(0) \eta'(0)
        + \frac{\partial^2 H}{\partial \eta^2} \eta'(0)^2 +
        \frac{\partial H}{\partial \alpha} \alpha''(0) +
        \frac{\partial H}{\partial \eta} \eta''(0) \\
        = & D^2H(v(0),v(0)) + \nabla H \cdot (v'(0), v'(0)) \\
        = & D^2H(v(0),v(0)) + \lambda \nabla K \cdot (v'(0), v'(0))
    \end{aligned}
\end{equation}
where we denote $D^2H(v,v)$ to be the quadratic form arising from the Hessian
matrix of $H$ acting on vector $v$. $v(t) = (\alpha'(t), \eta'(t))$.
Recall $\lambda$ was defined in \eqref{eq lagrange multiplier}.
Note that $K(\alpha(t), \eta(t)) = c$.
Differentiating in $t$ gives us for any $t$,
\begin{equation}
    \nabla K \cdot (v'(t), v'(t)) = - D^2K(v(t), v(t))
\end{equation}
Thus
\begin{equation}\label{eq hessian test}
    \frac{\partial^2 H(\alpha(t), \eta(t))}{\partial t^2} =
    D^2H(v(0),v(0)) + \lambda D^2K(v(0),v(0))
\end{equation}

Plugging the Hessians of H and K as well as
$v(0) = (\frac{1}{2f(\eta)}, -1)$ into \eqref{eq hessian test}, we get
\begin{equation}
    \begin{aligned}
        \eqref{eq hessian test} 
        & = -\frac{1}{4\eta f(-\eta) F^2(-\eta)} -
        \frac{6\eta f(-\eta)}{F(-\eta)^2} - \frac{f(-\eta)^2}{F(-\eta)^3} +
        \frac{f(-\eta)}{F(-\eta)^2} + \frac{1}{F(-\eta)}
    \end{aligned}
\end{equation}

The expression above is negative for all $\eta > 0$.
We conclude this section with the following corollary.
\begin{cor}
The expected survival time $H(\alpha, \eta)$ obtains a constrained local
maximum at $\alpha = 0$ under the constraint $K(\alpha, \eta) = c$.
\end{cor}


Our result indicates that if we have a raw signal without autocorrelation that
is correlated with the target (with only one nonvanishing cross-correlation
term), attempting to smooth it with an EMA filter will only make the portfolio
less profitable in the sense that we cannot hope to reduce the expected
turnover of the portfolio while keeping the return at a fixed level by
smoothing the raw signal by a small amount.

The result in this paper does not apply to a large amount of smoothing,
although the authors strongly suspect that the raw signal will still
outperform smootheded signals by any amount.
We derive the following estimate in the absence of a global result.
\begin{cor}
Suppose $X_t$ is identically independently distributed Gaussian signals and
$Y_t$ is correlated with $X_t$ by $\rho$, as in \eqref{y_t}. For any
$\alpha > 0$, we may apply EMA to $X_t$ as in (\ref{defn EMA}), which results
in AR(1) signal $\Tilde{X}_t$. If we keep the threshold $\eta$ fixed, the
target-portfolio correlation with the raw signal $X_t$ outperforms that of the
smoothed signal $\Tilde{X_t}$, and the ratio of improvement
(compared to the correlation of the original AR(1) signal) is greater than
\begin{equation}
\mathcal{R}(\alpha, \eta) = \frac{1 - \sqrt{1-\alpha^2} + \sqrt{1-\alpha^2}
    F(\sqrt{\frac{1+\alpha}{1-\alpha}} \eta) -
    \sqrt{1-\alpha^2} F(\sqrt{\frac{1-\alpha}{1+\alpha}} \eta)}{\sqrt{1-\alpha^2}
    - \sqrt{1-\alpha^2} F(\sqrt{\frac{1+\alpha}{1-\alpha}} \eta)
    + \sqrt{1-\alpha^2} F(\sqrt{\frac{1-\alpha}{1+\alpha}} \eta)}
\end{equation}
\end{cor}

\begin{figure}[h]
    \centering
    \includegraphics[width=0.5\textwidth]{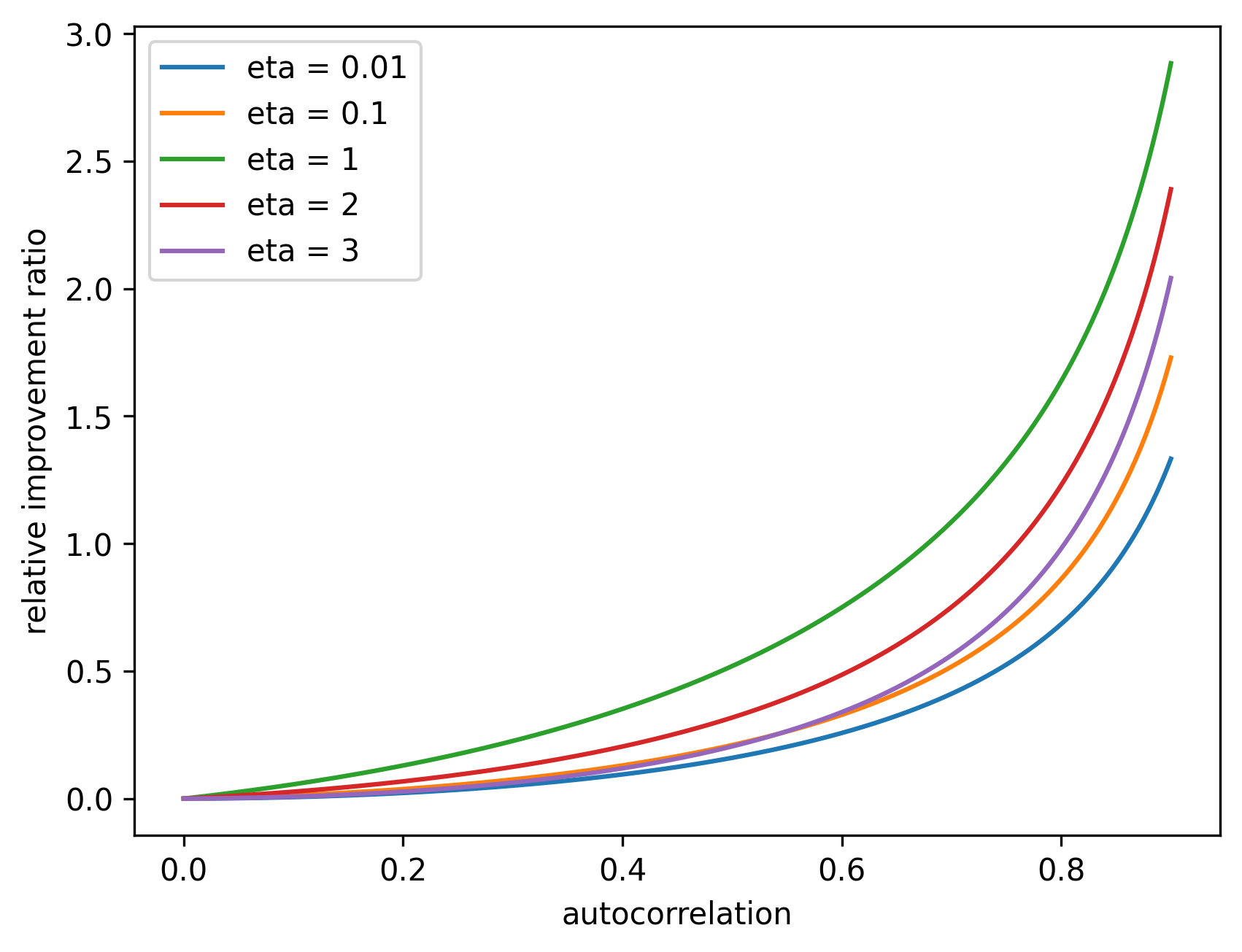}
    \caption{Improvement of signal target correlation by removing autocorrelation}
\end{figure}

\section{Directions for future investigation} \label{sec:future dir}
In the final part of the paper, we make a few comments on the directions for further research.

1. \emph{Multi-step signal target correlation}
In reality, the target is often correlated not only with the current signal,
but also with some lags of the current signal (some past values of the signal).
The authors are currently investigating the problem where we formulate it in
the Markov Decision Process framework and hope to derive a closed form optimal
policy depending on the transaction cost. Upon initial analysis, the form of
the optimal policy depends largely on the relative magnitude of the transaction
cost and signal strength.

2. \emph{Generalized portfolio model}
In this paper, we imposed several restrictions on the portfolio make-up, for
example the single asset condition and the two-state restriction. A natural
generalization is to allow the portfolio manager to hold any amount of shares
inside a continuous bounded interval, and to allow the manager to select from
multiple assets in constructing a portfolio.

\newpage

\appendix

\section{Existence and uniqueness of $h(x, \alpha, \eta)$} \label{AppendixA}

In this section, we show that the solution to
\eqref{integral eq for survival time}) exists and is differentiable.

Below we state the outline of our proof.
Let $T^{\alpha, \eta}$ denote the operator, which maps $C^{loc}(D)$ to itself,
by
\begin{equation*}
    T^{\alpha, \eta} g(x, \alpha) =
    1 + \int_{-\eta}^{+\infty} g(y, \alpha) \frac{1}{\sqrt{1 - \alpha^2}}
    f(\frac{y - \alpha x}{\sqrt{1 - \alpha^2}}) dy
\end{equation*}
A solution to \eqref{integral eq for survival time} is given by the fixed
point of $T^{\alpha, \eta}$, provided one exists.
If $T^{\alpha, \eta}$ is a family of contraction maps, then
lemma \ref{thm pde solution} is a direct result of the contraction mapping
theorem. However, for any $R > 0$, $T^{\alpha, \eta}$ fails to map the radius
$R$ ball in $C(D)$ into itself. In fact, given $g \in C(D)$, we have
\begin{equation*}
    \|T^{\alpha, \eta} g\|_{L^\infty_x} \leq 1 + \|g\|_{L^\infty_x}
    \sup_x \int_{-\eta}^{+\infty} \frac{1}{\sqrt{1 - \alpha^2}}
    f(\frac{y - \alpha x}{\sqrt{1 - \alpha^2}}) dy = 1 + \|g\|_{L^\infty_x},
\end{equation*}
and equality can be approached by a series of suitably chosen g's. \\

The problem is caused by the unboundedness of the expected survival time when
as $x$ approaches infinity. However, the expected survival time is locally
continuous and bounded. To overcome the obstacle, we define a family of
operators which approximates $T^{\alpha, \eta}$ locally and show that the
fixed points of these operators converge locally.

\begin{defn}\label{defn:Tfamily}
For $\beta > 0$, let $\phi_{\beta}(y)$ be the smooth decreasing function
defined by
\begin{equation}
    \phi_{\beta}(y) =
    \begin{cases}
        1 &y < 0 \\
        e^{-\beta y} &y > 1.
    \end{cases}
\end{equation}
Define the operator family
\begin{equation*}
    T^{\alpha, \eta, \beta}: C_\loc(D) \rightarrow C_\loc(D)
\end{equation*}
by
\begin{equation} \label{Talphabeta}
    T^{\alpha, \eta, \beta} g(x) = 1 + \int_{-\eta}^{+\infty} g(y)\phi_{\beta}(y)
    \frac{1}{\sqrt{1 - \alpha^2}} f(\frac{y - \alpha x}{\sqrt{1 - \alpha^2}}) dy.
\end{equation}
\end{defn}

We derive some estimates for $T^{\alpha, \eta, \beta}$ in the following
\begin{lem}
    Let $T^{\alpha, \eta, \beta}$ be as in Definition \ref{defn:Tfamily}. Then
    \begin{enumerate}
        \item $T^{\alpha, \eta, \beta}$ is bounded on $L^\infty_x((-\eta, +\infty))$,
            and we have
            \begin{equation}
                \|T^{\alpha, \eta, \beta} g\|_{L^\infty_x} \leq 1 + C(\alpha, \eta, \beta)\|g\|_{L^\infty_x}
            \end{equation}
            where $C(\alpha, \eta, \beta) < 1$ and
            $C(\alpha, \eta, \beta) \rightarrow 1$ as $\beta \rightarrow 0$.
        \item $T^{\alpha, \eta, \beta}$ is a contraction mapping which maps the ball
            \[
            B\left(0, \frac{1}{1 - C(\alpha, \eta, \beta)}\right) =
            \left\{g: \sup_x|g(x)| \leq \frac{1}{1 - C(\alpha, \eta, \beta)} \right\}
            \]
            into itself and satisfies
        \begin{equation}
            \|T^{\alpha, \eta, \beta} g_1 - T^{\alpha, \eta, \beta} g_2\|_{L^\infty_x}
            \leq C(\alpha, \eta, \beta) \|g_1 - g_2\|_{L^\infty_x}
        \end{equation}
    \end{enumerate}
\end{lem}

\begin{proof}
    We prove (1) first. We have
    \begin{equation}
        \|T^{\alpha, \eta, \beta} g\|_{L^\infty_x} \leq 1 + \|g\|_{L^\infty_x}
        \sup_x \int_{-\eta}^{+\infty} \phi_{\beta}(y)\frac{1}{\sqrt{1 - \alpha^2}}
        f(\frac{y - \alpha x}{\sqrt{1 - \alpha^2}}) dy.
    \end{equation}
    Let 
    \begin{equation*}
        C(\alpha, \eta, \beta) = \sup_x \int_{-\eta}^{+\infty}
        \phi_{\beta}(y)\frac{1}{\sqrt{1 - \alpha^2}}
        f(\frac{y - \alpha x}{\sqrt{1 - \alpha^2}}) dy
    \end{equation*}
    Then (1) follows straightforwardly and (2) follows directly from (1).
    \end{proof}

\begin{lem}\label{lem log bound}
We prove an apriori bound for the fixed point of $T^{\alpha, \eta}$:
For fixed $-1 < \alpha < 1$, $\eta > 0$, there exist constants
$A_1(\alpha, \eta)$, $A_2(\alpha, \eta)$ and $R(\alpha, \eta) > 0$, such that
if g(x) satisfies the logarithmic growth bound
\begin{equation}\label{log growth bound}
    g(x) \leq
    \begin{cases}
        A_1(\alpha, \eta), & for\ -\eta < x < R(\alpha, \eta) \\
        A_1(\alpha, \eta) + A_2(\alpha, \eta)\log x & for\ x \geq R(\alpha, \eta),
    \end{cases}
\end{equation}
then $T^{\alpha, \eta} g(x)$ also satisfies the same estimates \eqref{log growth bound}.
\end{lem}
Moreover, the constants $A_1(\alpha, \eta), A_2(\alpha, \eta) $ and
$R(\alpha, \eta)$ are uniformly bounded for $(\alpha, \eta)$ in compact sets.
\begin{proof}
Assume $g(x) \in C_\loc(\R)$ satisfies the bound \eqref{log growth bound} for
$A_1, A_2$ and $R$ to be determined.
For $x \geq R(\alpha, \eta)$, let $\delta>0$ be a small constant which does not
depend on $\alpha, \eta$ (to be determined). We estimate $Tg(x)$ by
\begin{equation}
    \begin{aligned}
        T^{\alpha, \eta}g(x)
        & = 1 + \int_{-\eta}^{+\infty} g(y) \frac{1}{\sqrt{1 - \alpha^2}}
        f(\frac{y - \alpha x}{\sqrt{1 - \alpha^2}}) dy \\
        & = 1 + \left[ \int_{((1 - \delta)\alpha x)}^{(1 + \delta)\alpha x} +
        \int_{-\eta}^{(1 - \delta)\alpha x} +
        \int_{(1 + \delta)\alpha x}^{+\infty }\right] g(y)
        \frac{1}{\sqrt{1 - \alpha^2}} f(\frac{y - \alpha x}{\sqrt{1 - \alpha^2}}) dy \\
        & = 1 + I + II + III
    \end{aligned}
\end{equation}
where I, II, III are the three truncated integrals in order.
Now we estimate them separately.

For I, we have $y \in [(1 - \delta)\alpha x, (1 + \delta)\alpha x]$, and
\begin{equation*}
    \begin{aligned}
        g(y)
        & \leq A_1(\alpha, \eta) +
        A_2(\alpha, \eta) \log(1 + \delta)\alpha x \\
        & = A_1(\alpha, \eta) +
        A_2(\alpha, \eta) \log(1 + \delta)\alpha + A_2(\alpha, \eta) \log x.
    \end{aligned}
\end{equation*}
Thus
\begin{equation}
    \begin{aligned}
        I
        & \leq A_1(\alpha, \eta)
        \int_{((1 - \delta)\alpha x)}^{(1 + \delta)\alpha x}
        \frac{1}{\sqrt{1 - \alpha^2}}
        f(\frac{y - \alpha x}{\sqrt{1 - \alpha^2}}) dy +
        A_2(\alpha, \eta)\log (1 + \delta)\alpha + A_2(\alpha, \eta) \log x.
    \end{aligned}
\end{equation}
For II, we have
\begin{equation*}
    g(y) \leq A_1(\alpha, \eta) + A_2(\alpha, \eta) \log \alpha x
\end{equation*}
Thus
\begin{equation}\label{eq II x>eta}
    II \leq (A_1(\alpha, \eta) + A_2(\alpha, \eta)\log \alpha x)
    \cdot P(\epsilon \leq -\frac{\delta \alpha x}{\sqrt{1 - \alpha^2}})
\end{equation}
where $\epsilon \sim \mathcal{N}(0,1)$ and $P(\cdot)$ is its associated probability.\\
Similarly, for III we have
\begin{equation}\label{eq III x>eta}
    III \leq \int_{\delta \alpha x}^{+\infty} (A_1(\alpha, \eta) +
    A_2(\alpha, \eta)\log\alpha y) \frac{1}{\sqrt{1 - \alpha^2}}
    f(\frac{y - \alpha x}{\sqrt{1 - \alpha^2}}) dy
\end{equation}
Now, we choose suitable values depending on $\alpha$ and $\eta$ for
$A_1(\alpha, \eta)$, $A_2(\alpha, \eta)$ and $R(\alpha, \eta)$ in order to have
\begin{equation}
    1 + I + II + III \leq A_1(\alpha, \eta) + A_2(\alpha, \eta) \log x.
\end{equation}
Denote the first and second term on the RHS of \eqref{eq II x>eta}) by $II_1$
and $II_2$ respectively, and similarly for $III_1$ and $III_2$ for
\eqref{eq III x>eta}). We now estimate the sum $1 + I + II_1 + III_1$.
Since $\log(1 + \delta \alpha) < 0$, we can choose $A_2(\alpha, \eta)$ large
enough, such that
\begin{equation}
    1 + A_2(\alpha, \eta) \log(1 + \delta)\alpha < 0.
\end{equation}
Thus
\begin{equation}\label{eq 1 log bound}
    \begin{aligned}
        1 + I + II_1 + III_1
        & \leq A_1(\alpha, \eta) + A_2(\alpha, \eta)\log x + 1 +
        A_2(\alpha, \eta) \log(1 + \delta)\alpha \\
        & < A_1(\alpha, \eta) + A_2(\alpha, \eta)\log x
    \end{aligned}
\end{equation}
Noting that the inequality \eqref{eq 1 log bound} is strict, and that $II_2$
and $III_2$ decay in x, we now choose $R(\alpha, \eta)$ large enough, which
causes $II_2 + III_3$ to be small enough such that
\begin{equation}
    1 + I + II_1 + III_1 + III_2 + III_2 <
    A_1(\alpha, \eta) + A_2(\alpha, \eta)\log x
\end{equation}
is still true.
Therefore the conclusion is proved for $x > R(\alpha, \eta)$.

It remains to prove the conclusion for $-\eta < x < R(\alpha, \eta)$.
For such x, we estimate $|Tg(x)|$ by
\begin{equation}\label{eq log x<R}
    \begin{aligned}
        |T^{\alpha, \eta}g(x)|
        & \leq 1 + \left[ \int_{-\eta}^{R(\alpha, \eta)} +
        \int_{R(\alpha, \eta)}^{+\infty} \right]
        g(y)\frac{1}{\sqrt{1 - \alpha^2}}
        f(\frac{y - \alpha x}{\sqrt{1 - \alpha^2}}) dy \\
        & = 1 + IV + V
    \end{aligned}
\end{equation}
where IV and V are the two truncated integrals in \eqref{eq log x<R}.\\
For $-\eta < y < R(\alpha, \eta)$,
\begin{equation*}
    g(y) = A_1(\alpha, \eta).
\end{equation*}
We estimate IV by
\begin{equation*}
    IV \leq A_1(\alpha, \eta)\int_{-\eta}^{R(\alpha, \eta)}
    \frac{1}{\sqrt{1 - \alpha^2}} f(\frac{y - \alpha x}{\sqrt{1 - \alpha^2}}) dy.
\end{equation*}
For V, we do
\begin{equation}\label{eq 2 log bound}
    \begin{aligned}
        V
        & \leq A_1(\alpha, \eta) \int_{R(\alpha, \eta)}^{+\infty}
        \frac{1}{\sqrt{1 - \alpha^2}}
        f(\frac{y - \alpha x}{\sqrt{1 - \alpha^2}}) dy \\
        & + A_2(\alpha, \eta) \int_{R(\alpha, \eta)}^{+\infty} \log(y)
        \frac{1}{\sqrt{1 - \alpha^2}} f(\frac{y -\alpha x}{\sqrt{1 - \alpha^2}}) dy \\
    \end{aligned}
\end{equation}
Denote the two terms on the RHS of \eqref{eq 2 log bound} by $V_1$ and $V_2$.
\begin{equation}\label{eq 3 log bound}
    1 + IV + V_1 \leq 1 +
    A_1(\alpha, \eta) P(\sqrt{1 - \alpha^2} \epsilon \leq -\alpha R(\alpha, \eta) - \eta).
\end{equation}
There exists $A_1(\alpha, \eta)$ large enough depending on $\alpha$, $\eta$
and our previous choice of $R(\alpha, \eta)$, such that the RHS of
\eqref{eq 3 log bound} is smaller than $A_1(\alpha, \eta)$.
We take $A_1(\alpha, \eta)$ to be even larger if necessary, depending on
$A_2(\alpha, \eta)$ and $R(\alpha, \eta)$, such that the bound still holds
after adding $V_2$.
\begin{rem}
The constants are well defined, as the choice of $A_2(\alpha, \eta)$ depends
only on $\alpha, \eta$ and then the choice of $R(\alpha, \eta)$ depends on
$\alpha$, $\eta$ and $A_2(\alpha, \eta)$. Finally, the choice of
$A_1(\alpha, \eta)$ depends on $\alpha$, $\eta$ and the previous choices of the
two other constants.
\end{rem}
The proof is concluded.
\end{proof}

\begin{lem}\label{lem log bound beta}
The result of Lemma \ref{lem log bound} is also true for $T^{\alpha, \eta, \beta}$
with $\beta > 0$, with constants $A_1(\alpha, \eta)$, $A_2(\alpha, \eta)$ and
$R(\alpha, \eta)$ uniform in $\beta$.
\end{lem}
\begin{proof}
This follows from the fact that for any $\beta > 0$,
it holds that $T^{\alpha, \eta, \beta} g \leq T^{\alpha, \eta} g$
for positive functions $g$.
Thus the proof of Lemma \ref{lem log bound} applies to
$T^{\alpha, \eta, \beta}$ without needing to adjust constants.
\end{proof}
Denote by $h^\beta(x, \alpha, \eta)$ the solution to the equation
\begin{equation*}
    T^{\alpha, \eta, \beta}g(x) = g(x)
\end{equation*}
Then, using Lemma \ref{lem log bound beta}, we get the logarithmic growth bound
for the solutions:
\begin{cor}
The function $h^\beta(x, \alpha, \eta)$ is locally $C^1$ and satisfies
\begin{equation}\label{log bound for solution}
    h^\beta(x, \alpha, \eta) \leq
    \begin{cases}
        A_1(\alpha, \eta),  &for\ -\eta < x < R(\alpha, \eta) \\
        A_1(\alpha, \eta) + A_2(\alpha, \eta)\log x  &for\ x \geq R(\alpha, \eta)
    \end{cases}
\end{equation}
uniformly in $\beta > 0$.
\end{cor}
\begin{proof}
Let $\tB$ denote the set of continuous functions which satisfy the growth
bound \eqref{log bound for solution}). Then $\tB$ is closed under uniform
convergence. Thus $T^{\alpha, \eta, \beta}$ is a contraction map on
$\tB \cap B(0, \frac{1}{1 - C(\alpha, \eta, \beta)})$, and so the solution
satisfies \eqref{log bound for solution} as desired.
\end{proof}
Now we derive the solution to \eqref{eq for partial alpha of survival time}
from subsequence convergence of $h^\beta$.
\begin{lem}\label{lem diagonal argument}
There exists a subsequence $\beta_j \rightarrow 0$ such that, if we denote
$h_j = h^{\beta_j}$ for convenience, then there exists some function
$h(x, \alpha, \eta) \in C^{loc}(D)$ such that
$h_j(x, \alpha, \eta) \rightarrow h(x, \alpha, \eta)$ on $D$ pointwise,
and the convergence is locally uniform in the sense that for any (relatively)
compact subset $D' \subset D $,
$h_j \rightarrow h$ on $D'$ uniformly in $\beta$.
\end{lem}
\begin{proof}
We begin by using Arzel\`{a}-Ascoli theorem to construct a locally uniformly
convergent subsequence. Then we use the standard diagonal argument to select a
subsequence that converges uniformly on any compact subset of $D$.
Fix $R > 0$ and $0 < \beta < 1$ and let $D' \subset D$ be such that x and
$\eta$ are bounded and $\alpha$ is bounded away from 1.
We claim that
\begin{itemize}
    \item $h^\beta$ is bounded on $D'$ uniformly in $\beta$
    \item $h^\beta$ is equicontinuous on $D'$.
\end{itemize}
The uniform boundedness follows from Lemma \ref{lem log bound beta}.
Note that the constants $A_1(\alpha, \eta)$, $A_2(\alpha, \eta)$ and
$R(\alpha, \eta)$ are bounded on the compact interval
$-1 + \delta \leq \alpha \leq 1 - \delta$ and do not depend on $\beta$.
Therefore the boundedness is uniform in $\beta$.
We will prove equicontinuity by estimating the derivatives of
$h^\beta(x, \alpha, \eta)$:
\begin{equation}\label{equicontinuous for h^beta}
    \begin{aligned}
        \left| \frac{\partial h^\beta}{\partial x} \right|
        & = \left| \int_{-\eta}^{+\infty} h^\beta(y, \alpha, \eta) \phi_{\beta}(y)
        \frac{\partial}{\partial x}\left( \frac{1}{\sqrt{1 - \alpha^2}}
        f(\frac{y - \alpha x}{\sqrt{1 - \alpha^2}}) \right) dy \right| \\
        & \leq \int_{-\eta}^{+\infty} h^\beta(y, \alpha, \eta)
        \left|\frac{\partial}{\partial x}\left( \frac{1}{\sqrt{1 - \alpha^2}}
        f(\frac{y - \alpha x}{\sqrt{1 - \alpha^2}}) \right)\right| dy
    \end{aligned}
\end{equation}
By Lemma \ref{lem log bound beta} and the fact that
$\frac{1}{\sqrt{1 - \alpha^2}}f(\frac{y - \alpha x}{\sqrt{1 - \alpha^2}})$ and
its derivatives w.r.t $x$ and $\alpha$ decay exponentially in y, and that the
fact that the decay speed is uniform for $(x, \alpha) \in D'$, we conclude that
the integral \eqref{equicontinuous for h^beta} is bounded in $D'$ uniformly in
$\beta$. A similar argument shows uniform boundedness of
$\frac{\partial h^\beta}{\partial \alpha}$ and
$\frac{\partial h^\beta}{\partial \eta}$. Therefore, the conditions of the
Arzel\`{a}-Ascoli theorem are verified. Consequently, there exists a
subsequence $\beta_j \rightarrow 0$ such that
$h_j(x, \alpha, \eta) \Rightarrow h(x, \alpha, \eta)$ on $D'$. The argument
holds for any compact $D' \subset D$. Thus, a standard diagonal argument will
yield a subsequence $\beta_j$, such that the associated approximated solutions
$h_j(x, \alpha, \eta)$ converge uniformly to some function $h(x, \alpha, \eta)$
on any compact subset of $D$.
\end{proof}
    
\subsubsection{Proof of lemma \ref{thm pde solution}}
Now we are ready to prove lemma 1. We begin by showing that the limit
function $h$ obtained in Lemma \ref{lem diagonal argument} solves the equation
\begin{equation*}
    T^{\alpha, \eta} h(x, \alpha, \eta) = h(x, \alpha, \eta).
\end{equation*}
In fact, we have
\begin{equation}
    \begin{aligned}
        h(x, \alpha, \eta)
        & = \lim_{j} h_j(x, \alpha, \eta) \\
        & = \lim_{j} 1 + \int_{-\eta}^{+\infty} h_j(y, \alpha, \eta) \phi_{\beta}(y)
        \frac{1}{\sqrt{1 - \alpha^2}} f(\frac{y - \alpha x}{\sqrt{1 - \alpha^2}}) dy \\
        & = 1 + \int_{-\eta}^{+\infty} h(y, \alpha, \eta)
        \frac{1}{\sqrt{1 - \alpha^2}}f(\frac{y - \alpha x}{\sqrt{1 - \alpha^2}}) dy
    \end{aligned}
\end{equation}
The last equality holds by the Dominated Convergence Theorem, which is
applicable due to the logarithmic growth bound of $h$ and the exponential
decay of $f(\cdot)$.
This proves part 1 of the lemma. Part 2 follows from
Lemma \ref{lem log bound beta} and the construction of $h$.

It only remains to prove part 3, i.e. $h \in C^2(D)$.
We first show $h \in C^1(D')$. We need only to refine our convergent subsequence
$h_j$ such that the partial derivatives uniformly converge to some continuous
functions $D_xh$, $D_\alpha h$ and $D_\eta h$ in any compact subset
$D' \subset D$.
This can be done by following the exact same argument as for part 1 where we
prove the uniform boundedness and equicontinuity of the sequence.
Then we apply the Arzel\`{a}-Ascoli argument to
$\frac{\partial h_j}{\partial x}$ and $\frac{\partial h_j}{\partial \alpha}$
instead of $h_j$. After another diagonal argument, we will have:
on any compact subset $D' \subset D$,
\begin{equation*}
    \begin{aligned}
        & h_j \Rightarrow h \\
        & \frac{\partial h_j}{\partial x} \Rightarrow D_x h \\
        & \frac{\partial h_j}{\partial \alpha} \Rightarrow D_\alpha h \\
        & \frac{\partial h_j}{\partial \eta} \Rightarrow D_\eta h
    \end{aligned}
\end{equation*}
This proves $h$ is $C^1$ in D'. Similar argument proves $C^2$ differentiability.
This concludes the proof of lemma \ref{thm pde solution}.


\section{Partial derivatives of $h(x, \alpha, \eta)$}\label{AppendixB}

In this section, we compute derivatives of H
Recall that the expected survival time function $H(\alpha, \eta)$ is defined as
\begin{equation}\label{H def APD}
    H(\alpha, \eta) = \frac{1}{1 - F(\eta)} \int_\eta^{+\infty} h(x, \alpha, \eta) f(x) dx
\end{equation}
and $h(x, \alpha, \eta)$ is the solution to the integral equation
\begin{equation}\label{h def APD}
    h(x, \alpha, \eta) = 1 + \int_{-\eta}^{+\infty} h(\alpha, \eta, y) f_\alpha (x,y) dy
\end{equation}
where
\begin{equation*}
    f_\alpha (x,y) = \frac{1}{\sqrt{1 - \alpha^2}} f(\frac{y - \alpha x}{\sqrt{1 - \alpha^2}})
\end{equation*}
We compute the derivatives of $f_\alpha$ which is required later.
\begin{lem}\label{lem compute f alpha}
    \begin{enumerate}
        \item $ f_\alpha (x,y)|_{\alpha = 0} = f(y) $
        \item $\frac{\partial f_\alpha}{\partial \alpha}(x,y)|_{\alpha = 0} = xyf(y)$
        \item $\frac{\partial^2 f_\alpha}{\partial \alpha^2}(x,y)|_{\alpha = 0} =
            (x^2 - 1) (y^2 - 1) f(y) $
    \end{enumerate}
\end{lem}
\begin{proof}
    (1) is obvious. (2) follows from
    \begin{equation*}
        \begin{aligned}
            & \left. \frac{\partial}{\partial \alpha}
            \left( \frac{1}{\sqrt{1 - \alpha^2}} f(\frac{y - \alpha x}{\sqrt{1 - \alpha^2}}) \right) \right|_{\alpha = 0} \\
            = & \left. \frac{1}{1 - \alpha^2} \left[
            f'(\frac{y - \alpha x}{\sqrt{1 - \alpha^2}})\cdot
            \frac{\partial}{\partial \alpha} \left( \frac{y-\alpha x}{\sqrt{1-\alpha^2}} \right) \cdot
            \sqrt{1 - \alpha^2} - f(\frac{y - \alpha x}{\sqrt{1 - \alpha^2}}) \cdot
            \frac{\partial}{\partial \alpha} \sqrt{1 - \alpha^2}\right] \right|_{\alpha = 0} \\
            = & -xf'(y) \\
            = & xyf(y).
        \end{aligned}
    \end{equation*}
    where
    \begin{equation*}
        \left. \frac{\partial}{\partial \alpha}
        \left( \frac{y-\alpha x}{\sqrt{1-\alpha^2}} \right)
        \right|_{\alpha=0} = \left. \frac{-x\sqrt{1 - \alpha^2} - (y - \alpha x)
        \frac{-\alpha}{\sqrt{1 - \alpha^2}}}{1 - \alpha^2} \right|_{\alpha = 0} = -x
    \end{equation*}
    \begin{equation*}
        \left. \frac{\partial}{\partial \alpha} \sqrt{1 - \alpha^2}
        \right|_{\alpha=0} = \left. \frac{-\alpha}{\sqrt{1 - \alpha^2}}
        \right|_{\alpha=0} = 0
    \end{equation*}

    (3) follows from differentiating (2) (without evaluating $\alpha$ at 0).
\end{proof}

\begin{lem}
    Here we list some technical computations for quick review purpose:
    \begin{enumerate}
        \item $\int_u^{+\infty} xf(x) dx = f(u)$
        \item $\int_u^{+\infty} (x^2 - 1) f(x) dx = uf(u)$
        \item $\int_u^{+\infty} x^2 f(x) dx = uf(u) + 1 - F(u)$
    \end{enumerate}
    also, $F(u) = 1 - F(-u)$ which we used sometimes for simplification.
\end{lem}
\begin{proof}
    Note that xf(x)dx = -df(x), these follows from integration by parts.
\end{proof}

Now we are ready to compute the derivatives of h.
\begin{lem}\label{lem derivatives of h}
    We compute the derivatives of h up to order 2:
    \begin{enumerate}
        \item $h(0, \eta, x) = \frac{1}{F(-\eta)}$
        \item $\frac{\partial h}{\partial \alpha}(0, \eta, x) =
            \frac{f(-\eta)^2}{F(-\eta)^2} + \frac{f(-\eta)}{F(-\eta)} x $
        \item For $\frac{\partial^2 h}{\partial \alpha^2}(0,\eta, x)$, we have
            \begin{equation}
                \begin{aligned}
                    \frac{\partial^2 h}{\partial \alpha^2}(0,\eta, x) &
                    = \int_{-\eta}^{+\infty} \frac{\partial^2 h}{\partial \alpha^2}(0,\eta, y)
                    f(y) dy + 2x \int_{-\eta}^{+\infty}
                    \frac{\partial h}{\partial \alpha}(0,\eta,y) yf(y) dy \\
                    & + (x^2 - 1) \int_{-\eta}^{+\infty} h(0,\eta, y) (y^2 - 1) f(y) dy
                \end{aligned}
            \end{equation}
            and
            \begin{equation}
                \int_{-\eta}^{+\infty}
                \frac{\partial^2 h}{\partial \alpha^2}(0,\eta, y) f(y) dy =
                \frac{2f(-\eta)}{F(-\eta)} \left[ \frac{f(-\eta)^3}{F(-\eta)^2} -
                \frac{\eta f(-\eta)^2}{F(-\eta)} + \frac{f(-\eta)}{F(-\eta)} -
                f(-\eta) \right] + \frac{\eta^2 f(-\eta)^2}{F(-\eta)^2}
            \end{equation}
        \item In addition,
            \begin{equation}
                \int_{+\eta}^{+\infty}
                \frac{\partial^2 h}{\partial \alpha^2}(0,\eta,x)f(x) dx =
                4f(\eta) \left[ \frac{f(-\eta)^3}{F(-\eta)^2} -
                \frac{\eta f(-\eta)^2}{F(-\eta)} + \frac{f(-\eta)}{F(-\eta)}
                - f(-\eta) \right]
            \end{equation}
    \end{enumerate}
\end{lem}

\begin{proof}
We show (1) first. When $\alpha = 0$, $\tX_t$ is independent Gaussian signals.
Thus
    \begin{equation}
        \begin{aligned}
            h(0, \eta, x) & = \sum_{j=1}^{+\infty} j P(\tX_t\ crosses\
            -\eta\ at\ t=j\ for\ the\ first\ time) \\
            & = \sum_{j=1}^{+\infty} j (1 - F(-\eta))^{j-1} F(-\eta) \\
            & = \frac{1}{F(-\eta)}
        \end{aligned}
    \end{equation}
    To prove (2), we differentiate \eqref{h def APD} w.r.t $\alpha$ and we obtain
    \begin{equation}\label{eq for partial alpha of survival time}
        \frac{\partial h}{\partial \alpha} (\alpha, \eta, x) =
        \int_{-\eta}^{+\infty} \frac{\partial h}{\partial \alpha}
        (\alpha, \eta, y) f_{\alpha}(x,y) dy \\
        + \int_{-\eta}^{+\infty} h(\alpha,\eta,y)
        \frac{\partial f_{\alpha}}{\partial \alpha}(x,y) dy
    \end{equation}

    Evaluating at $\alpha = 0$ and using Lemma \ref{lem compute f alpha}, we get
    \begin{equation}
        \begin{aligned}
            \frac{\partial h}{\partial \alpha}(0,\eta,x) = \int_{-\eta}^{+\infty}
            \frac{\partial h}{\partial \alpha} (0, \eta, y) f(y) dy
            + \frac{f(-\eta)}{F(-\eta)} x
        \end{aligned}
    \end{equation}

    Multiplying the equation by f(x) and integrating from $-\eta$ to $+\infty$ yields
    \begin{equation}
        \int_{-\eta}^{+\infty} \frac{\partial h}{\partial \alpha}(0,\eta,x)
        f(x) dx = \frac{f(-\eta)^2}{F(-\eta)^2}
    \end{equation}
    This proves (2).
    To compute (3), we differentiate
    \eqref{eq for partial alpha of survival time} w.r.t $\alpha$. We obtain
    \begin{equation}
        \begin{aligned}
            \frac{\partial^2 h}{\partial \alpha^2}(\alpha,\eta,x) &
            = \int_{-\eta}^{+\infty}
            \frac{\partial^2 h}{\partial \alpha^2}(\alpha,\eta,y)
            f_\alpha (x,y) dy + 2 \int_{-\eta}^{+\infty}
            \frac{\partial h}{\partial \alpha}(\alpha,\eta,y)
            \frac{\partial f_\alpha}{\partial \alpha}(x,y) dy \\
            & + \int_{-\eta}^{+\infty} h(\alpha,\eta,y)
            \frac{\partial^2 f_\alpha}{\partial \alpha^2}(x,y) dy
        \end{aligned}
    \end{equation}
    Evaluating at $\alpha=0$ and using Lemma \ref{lem compute f alpha} yields
    \begin{equation}
        \frac{\partial^2 h}{\partial \alpha^2}(0,\eta,x) =
        \int_{-\eta}^{+\infty}
        \frac{\partial^2 h}{\partial \alpha^2}(0,\eta,y)f(y) dy +
        2x \int_{-\eta}^{+\infty}
        \frac{\partial h}{\partial \alpha}(0,\eta,y) yf(y) dy + (x^2 - 1)
        \int_{-\eta}^{+\infty} h(0,\eta,y) (y^2 - 1) f(y) dy
    \end{equation}
    Multiplying the equation above by f(x) and integrating from $-\eta$ to
    $+\infty$ yields
    \begin{equation}\label{eq D^2h}
        \begin{aligned}
            \int_{-\eta}^{+\infty}
            \frac{\partial^2 h}{\partial \alpha^2}(0,\eta,y)f(y) dy &
            = (1 - F(-\eta)) \int_{-\eta}^{+\infty}
            \frac{\partial^2 h}{\partial \alpha^2}(0,\eta,y)f(y) dy \\
            & + 2 \int_{-\eta}^{+\infty} xf(x) dx \cdot
            \int_{-\eta}^{+\infty} \frac{\partial h}{\partial \alpha}(0,\eta,y) yf(y) dy \\
            & + \int_{-\eta}^{+\infty} (x^2 - 1)f(x) dx \cdot
            \int_{-\eta}^{+\infty} h(0,\eta,y) (y^2 - 1)f(y) dy
        \end{aligned}
    \end{equation}
    Some quantities required are computed below:
    \begin{equation*}
        \begin{aligned}
            \int_{u}^{+\infty} xf(x) dx & = f(u) \\
            \int_{u}^{+\infty} (x^2 - 1) f(x) dx & = u f(u) \\
            \int_{-\eta}^{+\infty} \frac{\partial h}{\partial \alpha}(0,\eta,y) yf(y) dy
            & = \int_{-\eta}^{+\infty} \left( \frac{f(-\eta)^2}{F(-\eta)^2} +
            \frac{f(-\eta)}{F(-\eta)}y \right) yf(y) \\
            & = \frac{f(-\eta)^2}{F(-\eta)^2} f(-\eta) +
            \frac{f(-\eta)}{F(-\eta)} (-\eta f(-\eta) + 1 - F(-\eta) ) \\
            & = \frac{f(-\eta)^3}{F(-\eta)^2} -
            \frac{\eta f(-\eta)^2}{F(-\eta)} + \frac{f(-\eta)}{F(-\eta)} - f(-\eta)
        \end{aligned}
    \end{equation*}
    Plugging these into (\ref{eq D^2h}), we obtain (3).
    Finally, to obtain (4), we multiply (3) by f(x) and integrating
    from $+\eta$ to $+\infty$:
    \begin{equation}\label{eq D2h int}
        \begin{aligned}
            & \int_\eta^{+\infty}
            \frac{\partial h}{\partial \alpha^2}(0,\eta,x)f(x) dx \\
            = & (1-F(\eta)) \int_{-\eta}^{+\infty}
            \frac{\partial^2 h}{\partial \alpha^2}(0,\eta,y)f(y) dy +
            2 \int_{\eta}^{+\infty} xf(x) dx \cdot
            \int_{-\eta}^{+\infty} \frac{\partial h}{\partial \alpha}(0,\eta,y) yf(y) dy \\
            + & \int_{\eta}^{+\infty} (x^2-1)f(x) dx \cdot
            \int_{-\eta}^{+\infty} h(0,\eta,y) (y^2-1)f(y) dy
        \end{aligned}
    \end{equation}
    Plugging (3) into \eqref{eq D2h int}
    \begin{equation}
        (\ref{eq D2h int}) = 4f(\eta)
        \left[ \frac{f(-\eta)^3}{F(-\eta)^2} -
        \frac{\eta f(-\eta)^2}{F(-\eta)} + \frac{f(-\eta)}{F(-\eta)}
        - f(-\eta) \right]
    \end{equation}

Alternatively, we compute $\frac{\partial h}{\partial \eta}$ with a method
similar to what we did with $\frac{\partial h}{\partial \alpha}$. First, we
verify that (1) satisfies the equation \eqref{h def APD} by simply plugging (1)
into the equation. For $\frac{\partial h}{\partial \eta}$, differentiate the
equation w.r.t. $\eta$ and get
    \begin{equation}
        \frac{\partial h}{\partial \eta}(\alpha,\eta,x) =
        h(\alpha, \eta, -\eta+0)f_\alpha(-\eta+0) +
        \int_{-\eta}^{+\infty}
        \frac{\partial h}{\partial \eta}(\alpha,\eta,x)f_\alpha(x) dx
    \end{equation}
    Evaluating at $\alpha = 0$, we get
    \begin{equation}
        \frac{\partial h}{\partial \eta}(0,\eta,x) =
        \frac{1}{F(-\eta)} f(-\eta) + \int_{-\eta}^{+\infty}
        \frac{\partial h}{\partial \eta}(0,\eta,x) f(x) dx
    \end{equation}
    Multiply it by f(x) and integrate from $-\eta$ to ${+\infty}$, we get
    \begin{equation}
        \int_{-\eta}^{+\infty}
        \frac{\partial h}{\partial \eta}(0,\eta,x) f(x) dx =
        \frac{f(-\eta)}{F(-\eta)} (1-F(-\eta)) +
        \int_{-\eta}^{+\infty}
        \frac{\partial h}{\partial \eta}(0,\eta,x) f(x) dx \cdot (1-F(-\eta))
    \end{equation}
    which gives
    \begin{equation*}
        \frac{\partial h}{\partial \eta}(0,\eta,x) = \frac{f(-\eta)}{F(-\eta)^2}
    \end{equation*}

\end{proof}

\bibliographystyle{amsplain}
\bibliography{references}

\begin{bibdiv}
\begin{biblist}

\bib{bertsekas1996stochastic}{book}{
      author={Bertsekas, Dimitri},
      author={Shreve, Steven~E},
       title={Stochastic optimal control: the discrete-time case},
   publisher={Athena Scientific},
        date={1996},
      volume={5},
}

\bib{BiPaTa08}{article}{
      author={Bibbona, Enrico},
      author={Panfilo, Gianna},
      author={Tavella, Patrizia},
       title={The {O}rnstein--{U}hlenbeck process as a model of a low pass filtered white noise},
        date={2008},
     journal={Metrologia},
      volume={45},
      number={6},
       pages={S117},
}

\bib{BoBoDoGo18}{book}{
      author={Bouchaud, Jean-Philippe},
      author={Bonart, Julius},
      author={Donier, Jonathan},
      author={Gould, Martin},
       title={Trades, quotes and prices: financial markets under the microscope},
   publisher={Cambridge University Press},
        date={2018},
}

\bib{BoJe76}{book}{
      author={Box, George~EP},
      author={Jenkins, Gwilym~M},
      author={Reinsel, Gregory~C},
      author={Ljung, Greta~M},
       title={Time series analysis: forecasting and control},
   publisher={John Wiley \& Sons},
        date={2015},
}

\bib{boyd2017multi}{article}{
      author={Boyd, Stephen},
      author={Busseti, Enzo},
      author={Diamond, Steve},
      author={Kahn, Ronald~N},
      author={Koh, Kwangmoo},
      author={Nystrup, Peter},
      author={Speth, Jan},
      author={others},
       title={Multi-period trading via convex optimization},
        date={2017},
     journal={Foundations and Trends{\textregistered} in Optimization},
      volume={3},
      number={1},
       pages={1\ndash 76},
}

\bib{constantinides1979multiperiod}{article}{
      author={Constantinides, George~M},
       title={Multiperiod consumption and investment behavior with convex transactions costs},
        date={1979},
     journal={Management Science},
      volume={25},
      number={11},
       pages={1127\ndash 1137},
}

\bib{davis1990portfolio}{article}{
      author={Davis, Mark~HA},
      author={Norman, Andrew~R},
       title={Portfolio selection with transaction costs},
        date={1990},
     journal={Mathematics of operations research},
      volume={15},
      number={4},
       pages={676\ndash 713},
}

\bib{garleanu2013dynamic}{article}{
      author={G{\^a}rleanu, Nicolae},
      author={Pedersen, Lasse~Heje},
       title={Dynamic trading with predictable returns and transaction costs},
        date={2013},
     journal={The Journal of Finance},
      volume={68},
      number={6},
       pages={2309\ndash 2340},
}

\bib{genccay2001introduction}{book}{
      author={Gen{\c{c}}ay, Ramazan},
      author={Dacorogna, Michel},
      author={Muller, Ulrich~A},
      author={Pictet, Olivier},
      author={Olsen, Richard},
       title={An introduction to high-frequency finance},
   publisher={Elsevier},
        date={2001},
}

\bib{granger2014forecasting}{book}{
      author={Granger, Clive William~John},
      author={Newbold, Paul},
       title={Forecasting economic time series},
   publisher={Academic press},
        date={2014},
}

\bib{hamilton2020time}{book}{
      author={Hamilton, James~D},
       title={Time series analysis},
   publisher={Princeton university press},
        date={2020},
}

\bib{Ha89}{article}{
      author={Harvey, Andrew~C},
       title={Forecasting, structural time series models and the {K}alman filter},
        date={1990},
}

\bib{magill1976portfolio}{article}{
      author={Magill, Michael~JP},
      author={Constantinides, George~M},
       title={Portfolio selection with transactions costs},
        date={1976},
     journal={Journal of economic theory},
      volume={13},
      number={2},
       pages={245\ndash 263},
}

\bib{DeMeNo16}{article}{
      author={Mei, Xiaoling},
      author={DeMiguel, Victor},
      author={Nogales, Francisco~J},
       title={Multiperiod portfolio optimization with multiple risky assets and general transaction costs},
        date={2016},
     journal={Journal of Banking \& Finance},
      volume={69},
       pages={108\ndash 120},
}

\bib{merton1969lifetime}{article}{
      author={Merton, Robert~C},
       title={Lifetime portfolio selection under uncertainty: The continuous-time case},
        date={1969},
     journal={The review of Economics and Statistics},
       pages={247\ndash 257},
}

\bib{merton1975optimum}{incollection}{
      author={Merton, Robert~C},
       title={Optimum consumption and portfolio rules in a continuous-time model},
        date={1975},
   booktitle={Stochastic optimization models in finance},
   publisher={Elsevier},
       pages={621\ndash 661},
}

\bib{muth1960optimal}{article}{
      author={Muth, John~F},
       title={Optimal properties of exponentially weighted forecasts},
        date={1960},
     journal={Journal of the american statistical association},
      volume={55},
      number={290},
       pages={299\ndash 306},
}

\bib{PrFeWe21}{book}{
      author={Prado, Raquel},
      author={West, Mike},
       title={Time series: modeling, computation, and inference},
   publisher={Chapman and Hall/CRC},
        date={2010},
}

\bib{priestley1988non}{article}{
      author={Priestley, Maurice~Bertram},
       title={Non-linear and non-stationary time series analysis},
        date={1988},
     journal={London: Academic Press},
}

\bib{raudys2013moving}{inproceedings}{
      author={Raudys, Aistis},
      author={Len{\v{c}}iauskas, Vaidotas},
      author={Mal{\v{c}}ius, Edmundas},
       title={Moving averages for financial data smoothing},
organization={Springer},
        date={2013},
   booktitle={Information and software technologies: 19th international conference, icist 2013, kaunas, lithuania, october 2013. proceedings 19},
       pages={34\ndash 45},
}

\bib{zumbach2001operators}{article}{
      author={Zumbach, Gilles},
      author={M{\"u}ller, Ulrich},
       title={Operators on inhomogeneous time series},
        date={2001},
     journal={International Journal of Theoretical and Applied Finance},
      volume={4},
      number={01},
       pages={147\ndash 177},
}

\end{biblist}
\end{bibdiv}

\end{document}